\documentclass[12pt,english]{article}
\usepackage{babel}
\usepackage[cp1250]{inputenc}
\usepackage[T1]{fontenc}
\usepackage{amsfonts}
\usepackage{amsmath}
\usepackage{amsthm}
\usepackage{mathrsfs}
\usepackage{mathabx}
\usepackage{mathtools}
\usepackage{hyperref}
\usepackage{comment}
\usepackage{enumerate}
\usepackage{graphicx}
\usepackage{pictex}
\usepackage{locHs_sty_220913}

\title{Local well-posedness of Kolmogorov's two-equation model of turbulence in fractional Sobolev Spaces}

\author{Przemys\l aw Kosewski \footnote{Department of Mathematics and Information Sciences, Warsaw University of Technology, ul. Koszykowa 75,
00-662 Warsaw, Poland, E-mail address: przemyslaw.kosewski.dokt@pw.edu.pl} }

\begin{document}
\maketitle

\begin{abstract}
We study Kolmogorov's two-equation model of turbulence on $d-$dimensional torus. First, the local existence of the solution with the initial data from non-homogeneous fractional Sobolev spaces (Bessel potential spaces) $H^s$ with $s>\frac{d}{2}$ is proven using energy methods. Next, we show that solutions are unique in the class of solutions guaranteed by the local existence theorem.   
\end{abstract}

\vspace{0.3cm}

\no Keywords: Kolmogorov's two-equation model of turbulence, local well-posedness, uniqueness, fractional Sobolev space, Bessel potential.

\vspace{0.2cm}


\section{Introduction}
In 1941 A. N. Kolmogorov proposed a system of equations describing turbulent motion (see: \cite{Kolmog}, \cite{Spal} for discussion of the proposed system):
\begin{align}
\nabla \cdot v & = 0, \label{divEq} \\
v_t + v \nab v - \divv \left( \bno D v \right) & = - \nabla p , \label{vEq} \\
\om_t + v \nab \om - \divv \left( \bno \nab \om \right) & = - \alpha \om^2, \label{omEq} \\
b_t + v \nab b - \divv \left( \bno \nab b \right) & = - b \om + \bno \dvk , \label{bEq} 
\end{align}
where $v$-mean velocity, $b$-mean turbulent energy, $\omega$-dissipation rate. We equip the system with initial data:
\eqnsl{
v(0,x) = v_0(x), ~~~~
\om(0,x) = \om_0(x), ~~~~
b(0,x) =  b_0(x).
}{init}
Nowadays, the model is not used in engineering practice, however, the ones used share with Kolmogorov's model underlying ideas (which manifest as similar dissipation/source terms included in the system). More information regarding turbulence models (e.g. $k-\varepsilon$, $k-\omega$) can be found in \cite{RANS4}, \cite{RANS7}, \cite{TurbMod}, \cite{WILCOX}.
Recently, the research concerning mathematical analysis of Kolmogorov's model has accelerated. 
In \cite{BuM} authors showed the existence of a weak solution to Kolmogorov's turbulence model. It relies on the introduction of a new variable $E = |v|^2/2 + b$ representing total energy in the system. It allows for the replacement of $b$-equation with an equation depending on $E$, for which passage to the limit with the approximate solution is plausible. 
In \cite{MiNa} authors showed the existence of a weak solution, however, without recovering the equation for $b$. Instead, they obtain inequality.   
In \cite{KoKu} authors showed the existence of the strong, unique solution on some small time interval given initial data from $H^2(\mathbb{T}^3)$. 
In \cite{FaGra} and \cite{FaGra2} authors consider the 1D system constructed based on Kolmogorov's system structure. They showed that solutions exist even for initial data for which the diffusion coefficient vanishes. Also, they proved the existence of a class of smooth initial data, for which finite-time blow-up occurs. 
In \cite{KoKu2} authors showed that provided some smallness condition on initial data global strong solutions exist. 

\noindent
The main purpose of this article is to generalise the statement given in \cite{KoKu} by lowering the regularity requirement for initial data and by considering the $d$-dimensional spacial domain. Before we precisely state the result we need to introduce the notion of the solution to the system (\ref{divEq})-(\ref{init}). 
We say that functions $(v,\om,b)$ solve system (\ref{divEq})-(\ref{bEq}) in the weak sense if 
\begin{align}
\n{\partial_t v, w} + \n{v \cdot \nab v, w} + \left(  \frac{b}{\omega} D v, Dw \right) & = 0 
& & \forall w \in H^1_{\divv} (\td),  
\label{vEqW}\\
\n{\partial_t \om, z} + \n{ v \cdot \nab \om, z}  + \left( \frac{b}{\omega} \nab \om, \nabla z \right) 
& = - \alpha \n{\om^2, z}
& & \forall z \in H^1 (\td), 
\label{omEqW}\\
\n{\partial_t b, q} + \n{v \cdot \nab b, q} + \left( \frac{b}{\omega} \nab b, \nabla q \right) 
& = -  \n{ b  \om, q} + \n{\frac{b}{\omega} |D v |^2, q}
& & \forall q \in H^1 (\td)
\label{bEqW}
\end{align}
for a.a. $t \in (0,T)$. Now we can formulate the main theorem.
\begin{theorem} \label{Th1}
Let $d \in \mathbb{N}_{\geq 2}$, $s > d/2$. Let $(v_0,\om_0,b_0) \in \left( \hs \right)^{d+2}$, such that $\divv v_0 = 0$, $\min_{x \in \td} b_0(x) > 0$ and $\min_{x \in \td} \om_0(x) > 0$. 
Additionally, let time $T>0$ be such that
\eqns{
(1 - 2^{-\beta + 1}) \left( 1 + \nhs{v_0, \om_0, b_0}^2 \right)^{-\beta + 1}
 = (\beta - 1) \int_0^T C(\bmi, \omi, s, \tau) d \tau, 
}
where
$C(\omi, \bmi, s, \tau)$ is rational function finite for $\tau \geq 0$ (see (\ref{timeDepConstant})) and $ \beta = \beta(s)>1$. 
Then, there exists $t^*>T$, such that the system (\ref{vEqW})-(\ref{bEqW}), (\ref{init}) has unique solution $(v,\om,b)$ on $[0,t^*)$ such that 
\begin{align}
v & \in \left( C([0,t^*), H^s_{\divv}(\td)) \cap L^2(0,t^*,H^{s+1}_{\divv}(\td)) \cap W^{1,2}(0,t^*,H^{s-1}_{\divv}(\td)) \right)^{d}, \\
(\om,b) & \in \left( C([0,t^*), \hs) \cap L^2(0,t^*,\hsj) \cap W^{1,2}(0,t^*,H^{s-1}(\td)) \right)^{2}.
\end{align}
\end{theorem}
\begin{theorem} \label{Th1b}
Let $d \in \mathbb{N}_{\geq 2}$, $s > d/2$. Let $(v_0,\om_0,b_0) \in \left( \hs \right)^{d+2}$, such that $\divv v_0 = 0$, $\min_{x \in \td} b_0(x) > 0$ and $\min_{x \in \td} \om_0(x) > 0$. Let $T>0$ and
\begin{align}
v^i & \in \left( C([0,T), H^s_{\divv}(\td)) \cap L^2(0,T,H^{s+1}_{\divv}(\td))  \cap W^{1,2}(0,T,H^{s-1}_{\divv}(\td) \right)^{d}, \\
(\om^i,b^i) & \in \left( C([0,T), \hs) \cap L^2(0,T,\hsj)  \cap W^{1,2}(0,T,H^{s-1}(\td) \right)^{2}
\end{align}
be two solutions to system (\ref{vEqW})-(\ref{bEqW}), (\ref{init}). Then
\eqns{
v^1 = v^2, ~~~~ \om^1 = \om^2, ~~~~ b^1 = b^2 ~~~~\forall (t,x) \in [0,T] \times \Omega.  
}
\end{theorem}

\section{Preliminaries}
Let us start with recalling the definitions of spaces set on $\td = [0,1)^d$.
\begin{de}[see Remark 3.1.5 in \cite{Ruzh2010} or Section 3.2 and 3.5 in \cite{Beal2013} ]
Let $\{ u_n \}_{n=1}^\infty \subset C^\infty(\td)$, $u \in C^\infty(\td)$. We say that $u_n \rightarrow u$ in $C^\infty(\td)$ if $\partial^\alpha u_n \rightarrow \partial^\alpha u$ uniformly for all $\alpha \in \mathbb{N}_0$.
By $\mathcal{D}^\prime(\td)$ we denote continuous linear functionals on $C^\infty(\td)$. 
\end{de}
\begin{de}[see Definition 3.1.6 in \cite{Ruzh2010}]
Let $S(\mathbb{Z}^d)$ denote the space of rapidly decaying functions $\mathbb{Z}^d \rightarrow \mathbb{C}$. That is, $\varphi \in S(\mathbb{Z}^d)$ if for any $k<\infty$ there exists a constant $C_{\varphi, k}$ such that 
\eqns{
|\varphi(\xi)| \le \frac{C_{\varphi, k}}{\n{1+ |\xi|^2}^{k/2}}.
}
The topology on $S(\mathbb{Z}^d)$ is given by the seminorms $p_k$, where $k \in \mathbb{N}_0$ and \eqns{
p_k(\varphi) = \sup_{\xi \in \mathbb{Z}^d} \n{1 + |\xi|^2}^{\frac{k}{2}} |\varphi(\xi)|.
}
Then, a sequence $\{ \varphi_n \}_{n=1}^\infty \subset S(\mathbb{Z}^d)$ converges to the function $\varphi \in S(\mathbb{Z}^d)$ iff 
\eqns{
p_k(\varphi_n - \varphi) \overset{n \rightarrow \infty }{\longrightarrow} 0 \quad \text{for all } k \in \mathbb{N}_0.
}
By $S^\prime(\mathbb{Z}^d)$ we denote continuous linear functionals on $S(\mathbb{Z}^d)$.
\end{de}
\begin{de}[see Definition 3.1.8 in \cite{Ruzh2010}]
Toroidal Fourier transform $\mathcal{F}_{\td} = (f \mapsto \hat{f}): C^\infty(\td) \rightarrow S(\mathbb{Z}^d)$ is defined by 
\eqns{
\hat{f}(\xi) = \int_{\td} f(x) e^{-2 \pi i x \cdot \xi} dx.
}
Inverse toroidal Fourier transform $\mathcal{F}_{\td}^{-1} = (h \mapsto \widecheck{h}): S(\mathbb{Z}^d) \rightarrow C^\infty(\td) $ is given by 
\eqns{
\widecheck{h}(x) = \sum_{\xi \in \mathbb{Z}^d} h(\xi) e^{ i 2 \pi x \cdot \xi}.
}  
\end{de}
\begin{de}[see Definition 3.1.27 in \cite{Ruzh2010}]
Fourier transform extends to the mapping $\mathcal{F}_{\td}: \mathcal{D}^\prime(\td) \rightarrow S^\prime(\mathbb{Z}^d)$ by the formula
\eqns{
\langle \hat{u}, \varphi \rangle := \langle u, \iota \circ \widecheck{\varphi} \rangle,
}
where $u \in \mathcal{D}^\prime(\td)$, $\varphi \in S(\mathbb{Z}^d)$ and $\iota$ is defined by $(\iota \circ \psi)(x) = \psi(-x)$.
\end{de}
\begin{de}\label{defJs}
Let $s \in \mathbb{C}$. Bessel potential $\ls$ on the torus is defined as follows
\eqns{
(\ls f)(x) = \sum_{k \in \mathbb{Z}^d} \left( 1 + 4 \pi^2 |k|^2 \right)^{s/2} e^{2 \pi i x \cdot k} \hat{f}(k), 
}
where $\hat{f}$ denotes Fourier transform of $f$.
\end{de}
\noindent Now, let us recall the definition of fractional inhomogeneous Sobolev spaces on torus $\hsp$.
\begin{de}[see chapter 3.5.4 in \cite{Schm1987}]
Let $s$ be real number, $p \in (1,\infty)$. The inhomogeneous Sobolev Space $\hsp$ is defined as follows
\eqns{
\hsp = \left\{ u \in D^{'}(\td): \nhsp{u}{s}{p} = \norm{\ls u}_p < \infty
\right\}.
}
\end{de}
\noindent
Also, based on orthogonality of $\{e^{2\pi k x}\}_{k \in \mathbb{Z}^d}$ in $L^2(\td)$ the following characterisation is valid
\eqns{
\hs = \left\{ u \in D^{'}(\td): \norm{u}^2_{H^s_2} = \sum_{k \in \mathbb{Z}^d} \left( 1 + 4 \pi^2 |k|^2 \right)^{s} \left| \hat{u}(k) \right|^2 < \infty
\right\}.
}
Moreover, we introduce the following notation
\eqns{
H^s_{\divv}(\td) = \left\{ u \in [H^s(\td)]^d: \divv u = 0 \right\}.
}
\noindent
To simplify further expressions we introduce the following notation:
\eqns{
\inhs{f}{g} = \n{\ls f, \overline{\ls g}}_{L^2(\td)}.
} 
\noindent
Now, we will recall some know facts regarding fractional Sobolev spaces on the torus.
\begin{lem}[see \cite{Card2019}, \cite{Graf2014b}, \cite{Kato1988}, \cite{Guli1996}] \label{Lem1}
Let $s \geq 0$ and $p\in (1,\infty)$, $p_1, p_2, q_1, q_2 \in (1,\infty]$ such that
\eqns{
\frac{1}{p} = \frac{1}{p_1} + \frac{1}{q_1} = \frac{1}{p_2} + \frac{1}{q_2}.
}
Let $f, g \in C^\infty(\td)$. 
Then there exists $C = C(s,d,p,p_1,p_2,q_1,q_2)$ independent of $f$ and $g$ such that the following inequality holds:
\eqns{
\norm{\ls(f g) }_{p} \le C \left( \norm{\ls f}_{p_1} \norm{g}_{q_1} + \norm{f}_{p_2} \norm{\ls g}_{q_2} \right).
}
\end{lem}
\begin{lem}[see chapter 2.8.3 in \cite{Trie1983}] \label{Lem1.5}
    Let $p \in (1,\infty)$, $s > d/p$ and $f,g \in H^{s}_p(\td)$. Then, $fg \in H^{s}_p(\td)$ and there exists a constant $C>0$, independent of $f$ and $g$ such that
    \eqns{
        \norm{fg}_{H^{s}_p(\td)} \le C \norm{f}_{H^{s}_p(\td)} \norm{g}_{H^{s}_p(\td)}.
    }
\end{lem}
\begin{lem} [See Lemma 2.5(ii) in \cite{Cira2019}] \label{Lem6}
Let $s > \frac{d}{2}$ and $f \in \hs $. Then, function $f$ is continuous and there exists constant $C=C(s,d)$ independent of $f$ such that 
\eqns{
\nif{f} \le C \nhs{f}.
}
\end{lem}
\begin{lem}[see Theorem 5.5 in \cite{Beny2020} or Section 3.1 in \cite{Adam1991}] \label{Lem2}
Let $s > \frac{d}{2}$. Assume that $G$ is a smooth function on $\mathbb{R}$ with $G(0) = 0$. Then there exists $C$ independent of $f$ and $G$ such that:
\eqns{
\nhs{G(f)} \le C \ncs{G^{\prime}} \left( 1 + \nif{f} \right)^{\lceil s \rceil} \nhs{f}.
} 
\end{lem}
\begin{lem} \label{Lem2.5}
Let $s > \frac{d}{2}$. Assume that $G$ is a smooth function on $\mathbb{R}$ with $G^\prime(0) = 0$. Then there exists $C$ independent of $u$, $v$ and $G$ such that:
\eqns{
\nhs{G(u) - G(v)} \le 
C \ncs{G^{\prime \prime}} \nhs{u - v} 
\left( 1 + \nhs{u} + \nhs{v} \right)^{\lceil s \rceil + 1}.
} 
\end{lem}
\begin{proof}
The lemma is a direct consequence of Lemma \ref{Lem2}. Proceeding as in \cite{BaCheRa} Corollary 2.66, we see that
\eqns{
G(u) - G(v) = (u - v) \int_0^1 G^\prime(u + \tau (v - u)) d \tau, 
}
which can be understood classically due to $v$, $u$ both being continuous functions (see Lemma \ref{Lem6}). 
By applying $\hs$ norm to the both sides and using Lemma \ref{Lem1.5} we get 
\eqns{
\nhs{G(u) - G(v)} 
\le 
\nhs{u - v} \nhs{ \int_0^1 G^\prime(u + \tau (v - u)) d \tau}.
}   
Next, we may change the order of the norm and integral to get  
\eqns{
\nhs{G(u) - G(v)} 
\le 
\nhs{u - v} \int_0^1 \nhs{G^\prime(u + \tau (v - u))} d \tau.
}
As $G^\prime(0) = 0$ we may apply Lemma \ref{Lem2} to the term under integral
\eqns{
& \nhs{G(u) - G(v)}  \\
& \quad \quad \quad \le 
C \ncs{G^{\prime \prime}} \nhs{u - v} 
\int_0^1 
\left( 1 + \nif{u + \tau (v - u)} \right)^{\lceil s \rceil} \nhs{u + \tau (v - u)}
d \tau.
}
We may estimate the right-hand side using triangle inequality. We get   
\eqns{
& \nhs{G(u) - G(v)}  \\
& \quad \quad \quad \le 
C \ncs{G^{\prime \prime}} \nhs{u - v} 
\left( 1 + \nif{u} + \nif{v} \right)^{\lceil s \rceil} \n{\nhs{u} + \nhs{v}}.
}
By using Lemma \ref{Lem6} we obtain desired inequality.
\end{proof}
\begin{lem}\label{hsNormOfGradient}
Let $f: \td \rightarrow \mathbb{C}$ be such that $f \in H^{s+1}(\td)$. Then $\nhsj{f}^2 = \nhs{\nab f}^2 + \nhs{f}^2 $.
\end{lem}
\begin{proof}
By simple calculations we get the statement of the lemma:
\eqns{
\nhs{\nab f}^2 
& = 
\sum_{j=1}^d \sum_{k \in \mathbb{Z}^d} \left( 1 + 4 \pi^2 |k|^2 \right)^{s} \left| \widehat{\n{\frac{\partial f}{\partial x_j}}}(k) \right|^2 \\
& = 
\sum_{j=1}^d \sum_{k \in \mathbb{Z}^d} \left( 1 + 4 \pi^2 |k|^2 \right)^{s} 4 \pi^2 |k_j|^2 \left| \widehat{f}(k) \right|^2 \\
& = 
\sum_{k \in \mathbb{Z}^d} \n{\left( 1 + 4 \pi^2 |k|^2 \right)^{s+1} - \left( 1 + 4 \pi^2 |k|^2 \right)^{s} } \left| \widehat{f}(k) \right|^2 \\
&= \nhsj{f}^2 - \nhs{f}^2.
}
\end{proof}
\begin{lem} [See Lemma 2.5(i) in \cite{Cira2019}] \label{Lem10}
Let $p \in (1,\infty)$ and $\mu, \nu \in \mathbb{R}$ such that $\nu \le \mu$. Then $H^{\mu}_p(\td) \xhookrightarrow{} H^{\nu}_p(\td)$. 
\end{lem}
\begin{lem} [See Lemma 2.5(iii) in \cite{Cira2019}] \label{Lem11}
Let $p,q \in (1,\infty)$ and $\mu, \nu \in \mathbb{R}$ such that $\nu \le \mu$ and  
\eqns{
\mu - \frac{d}{p} = \nu - \frac{d}{q}.
}
Then $H^{\mu}_p(\td) \xhookrightarrow{} H^{\nu}_q(\td)$. 
\end{lem}
\begin{lem}[see \cite{Kato1988}, Appendix I] \label{Lem7}
Suppose that $s>0$, $p, p_2, p_4 \in (1,\infty)$ and $p_1,p_3\in (1,\infty]$ such that 
\eqns{
\frac{1}{p} = \frac{1}{p_1} + \frac{1}{p_2} = \frac{1}{p_3} + \frac{1}{p_4}. 
}
Let $f, g \in C^\infty(\mathbb{T}^d)$, then there exists constant $C$ independent of $f$ and $g$ such that 
\eqns{
\norm{\left[ \ls, f\right]g}_{L^p(\mathbb{T}^d)} \le C \left(
\norm{\nab f}_{L^{p_1}(\mathbb{T}^d)} \norm{\lsmj g}_{L^{p_2}(\mathbb{T}^d)}
+
\norm{g}_{L^{p_3}(\mathbb{T}^d)} \norm{\ls f}_{L^{p_4}(\mathbb{T}^d)}
\right),
}
where $\left[ \ls, f\right]g := \ls(fg) - f \ls g$.
\end{lem}
\begin{rem}
The above result was proven in \cite{Kato1988} for the case of $\mathbb{R}^d$. In Appendix I we provide detailed proof in the case of $\mathbb{T}^d$.
\end{rem}
\begin{rem}
References of Lemmas \ref{Lem1.5} and \ref{Lem2}  are provided for $\mathbb{R}^d$ domain. By following argument presented in Section 2.3.1 of \cite{Cira2019}, those formulations can be adapted for $\td$ case by considering extension operator $ H^s_p(\td) \ni f \rightarrow \phi \widetilde{f} \in H^s_p(\mathbb{R}^d)$, where $\phi$ is smooth, compactly supported function defined on $\mathbb{R}^d$, such that $\phi_{ | [0,1)^d } = 1$ and
\eqns{
\widetilde{f}(x) = f(x - \lfloor x \rfloor),
}
where $\lfloor x \rfloor = (\lfloor x_1 \rfloor, \dots, \lfloor x_d \rfloor)$, ($\lfloor \cdot \rfloor$ - floor function).
Then, it is clear that
\eqnsl{
\norm{f}_{W^{k,p}(\td)} 
\le 
\norm{\phi \widetilde{f}}_{W^{k,p}(\mathbb{R}^d)}
\le
C \norm{f}_{W^{k,p}(\td)} .
}{rem2Eq1}
From complex interpolation (see e.g.: Theorem 2.6 in \cite{Luna2018}), we can deduce analogous inequality for the fractional spaces. 
\eqnsl{
\norm{f}_{H^{s}_{p}(\td)} 
\le 
\norm{\phi \widetilde{f}}_{H^{s}_{p}(\mathbb{R}^d)}
\le
C \norm{f}_{H^{s}_{p}(\td)} .
}{rem2Eq2}
Lastly, by considering Lemma \ref{Lem1} for function $\phi \widetilde{f}$ and using (\ref{rem2Eq2}) we get the needed statements. The proof of the Lemma \ref{Lem2} is more complicated and we shall give more details.  
\end{rem}
\begin{proof}
Let $\psi$, $\phi$ be smooth, compactly supported functions defined on $\mathbb{R}^d$, such that $\psi_{ | [0,1]^d} \equiv 1$ and $\phi_{ | \sup \psi} \equiv 1$.
Also let us observe that $\widetilde{G(f)} = G(\widetilde{f})$.
Now, using Lemma \ref{Lem1.5} and the fact that $G(0) = 0$ we can write 
\eqns{
\norm{\psi \widetilde{G(f)}}_{H^s(\mathbb{R}^d)}
& =
\norm{\psi G(\widetilde{f})}_{H^s(\mathbb{R}^d)}
=
\norm{\psi G(\phi \widetilde{f})}_{H^s(\mathbb{R}^d)}
\le C
\norm{\psi}_{H^s(\mathbb{R}^d)}
\norm{G(\phi \widetilde{f})}_{H^s(\mathbb{R}^d)}
\\ & \le C
\norm{G(\phi \widetilde{f})}_{H^s(\mathbb{R}^d)}
\le C 
\ncs{G^{'}} \left( 1 + \norm{\phi \widetilde{f}}_{L^\infty(\mathbb{R}^d)} \right)^{\lceil s \rceil} \norm{\phi \widetilde{f}}_{H^s(\mathbb{R}^d)}.
}  
We can easily estimate both sides using (\ref{rem2Eq2}) to get
\eqns{
\norm{G(f)}_{H^s(\td)}
 \le 
 C \ncs{G^{'}} \left( 1 + \norm{f}_{L^\infty(\td)} \right)^{\lceil s \rceil} \norm{f}_{\hs)}.
} 
\end{proof}
\begin{lem} \label{Lem9}
Let $f \in \hsj $ and $s > \frac{d}{2}$. Then we have 
\eqnsl{
\nif{\nab f} \le C
\norm{f}_{H^{s}(\Omega)}^{\frac{1}{2}\left( s - \frac{d}{2} \right)}
\norm{f}_{H^{s+1}(\Omega)}^{1-\frac{1}{2}\left( s - \frac{d}{2} \right)}
\text{~~for~} s \in \left( \frac{d}{2}, \frac{d}{2} + 1 \right]
}{Lem9A}
and
\eqns{
\nif{\nab f} \le C \nhs{f} \text{~~for~} s \in \left( \frac{d}{2} + 1, \infty  \right).
} 
\end{lem}
\begin{proof}
First we concentrate on the case $s \in \left( \frac{d}{2}, \frac{d}{2} + 1 \right]$. We see that based on Lemma \ref{Lem6} we have
\eqns{
\norm{\nabla f}_{L^\infty(\Omega)}
\le 
C\norm{\nabla f}_{H^{\frac{1}{2}\left( s + \frac{d}{2} \right)}(\Omega)}.
}
We see that
\eqns{
\frac{1}{2}\left( s + \frac{d}{2} \right) = \frac{1}{2}\left( s - \frac{d}{2} \right) (s-1) + \left(1-\frac{1}{2}\left( s - \frac{d}{2} \right) \right)s
}
and thus we may use interpolation inequality to get
\eqns{
\norm{\nabla f}_{L^\infty(\Omega)}
& \le C
\norm{\nabla f}_{H^{s-1}(\Omega)}^{\frac{1}{2}\left( s - \frac{d}{2} \right)}
\norm{\nabla f}_{H^{s}(\Omega)}^{1-\frac{1}{2}\left( s - \frac{d}{2} \right)}.
}
Thus based on Lemma \ref{hsNormOfGradient} we get
\eqns{
\norm{\nabla f}_{L^\infty(\Omega)}
\le C
\norm{f}_{H^{s}(\Omega)}^{\frac{1}{2}\left( s - \frac{d}{2} \right)}
\norm{f}_{H^{s+1}(\Omega)}^{1-\frac{1}{2}\left( s - \frac{d}{2} \right)}.  
}
If $s \in \left( \frac{d}{2} + 1, \infty \right)$ we have
\eqns{
\norm{\nabla f}_{L^\infty(\Omega)}
\le 
C_1\norm{\nabla f}_{H^{s-1}(\Omega)}
\le 
C_2\norm{f}_{H^{s}(\Omega)},
}
which follows from Lemma \ref{Lem6} and \ref{hsNormOfGradient}.
\end{proof}

\section{Proof of Theorem \ref{Th1}}
The proof will be divided into several steps to better present and simplify reasoning.
\subsection{Definitions of auxiliary functions}
Given initial data $\om_0$, $b_0$ we define $\bmi$, $\omi$, $\oma$ in the following way 
\eqq{\bmi = \min_{x \in \td} b_{0}(x) ,  }{f}
\eqq{
\omi = \min_{x \in \td} \om_{0}(x), 
~~~~~~~~
\oma = \max_{x \in \td} \om_{0}(x).
}{g}
We define auxiliary functions as follows
\eqq{
\bmt = \bmit,  \hd ~
\omt = \omitt, \hd ~
\ommt = \omat, \hd ~
\mnit = \frac{1}{4}\frac{\bmt}{\ommt}.  }{h}
To define an approximate problem we have to introduce a few auxiliary functions. For fixed $t>0$ we denote by $\Pht= \Pht(x)$ a smooth, non-decreasing function such that
\eqq{
\Pht(x)= \left\{
\begin{array}{rll}
\frac{1}{2}\bmt & \m{ for } & x<\frac{1}{2}\bmt \\ x &  \m{ for } & x\geq\bmt\\
\end{array}  \right. ,}{defPsi}
where  $\bmt$ is defined by (\ref{h}).  We assume that the function $\Pht $ also satisfies
\eqq{ 
|\Pht^{(k)}(x)|\leq c_{0} (\bmt)^{1 - k} ~~~\forall k \in \{ 1, \dots, \lceil s \rceil + 1 \}   
}{estiPsi}
where, $c_{0}$ is a constant independent on $x$ and $t$ (see in the appendix of \cite{KoKu} for details). We  also need smooth, non-decreasing function $\Pst$ such that
\eqq{ \Pst(x)=\left\{
\begin{array}{rll}
\frac{1}{2}\omt & \m{ for } & x< \frac{1}{2}\omt \\
x  & \m{ for } & x \in [ \omt, \ommt] \\
2\ommt  & \m{ for } & x >  2\ommt \\
\end{array} \right. ,}{defPhi}
We assume that this function additionally satisfy
\eqq{ 
|\Pst^{(k)}(x)|\leq c_{0} (\omt)^{1 - k} ~~~\forall k \in \{ 1, \dots, \lceil s \rceil + 1 \}, 
}{estiPhi}
for some constant $c_{0}$ (the construction of $\Pst$ is similar to argument presented in the appendix of \cite{KoKu}).
\subsection{Approximated system}
To obtain the approximate system we will follow the procedure used in \cite{Wang2013}, \cite{Wang2011}. Let us define the operator $\pn$ in the following way 
\eqnsl{
\pn f(x) = \sum_{|k| < n} f_k e^{2 \pi i k x}, ~~~~~
f_k = \int_{\mathbb{T}^d} f(x) e^{-2 \pi i k x} dx .
}{operatorPn}
In later parts, we will require some properties of the $\pn$ operator. First, it is obvious that
\eqns{
\pn \pn = \pn ,
} 
which follows from orthogonality in $L^2(\td)$ of functions $\{ e^{2 \pi i k x} \}_{k \in \mathbb{Z}^d}$. Let us define $C(n) = \sum_{|k| < n} 1$ and observe that for $m \in \mathbb{N}$ and $1\le p \le \infty$ we have
\eqns{
\norm{\frac{\partial^m  \n{\pn f}}{\partial x_{i_1} \dots \partial x_{i_m}} }_p 
& = 
(2 \pi)^{m} \norm{\sum_{|k| < n} k_{i_1} \dots k_{i_m} f_k e^{2 \pi i k \cdot}}_p 
 \le 
(2 \pi)^{m} \sum_{|k| < n} |k_{i_1} \dots k_{i_m}| |f_k| \norm{e^{2 \pi i k \cdot}}_p \\ &
\le 
(2 \pi n)^{m} \sum_{|k| < n}  \left| f_k \right|
\le 
(2 \pi n)^{m} \sqrt{C(n)} \bigg(\sum_{|k| < n} \left| f_k \right|^2 \bigg)^\jd \\
& = 
(2 \pi n)^{m} \sqrt{C(n)} \nd{\pn f}.
}
Thus for $m \in \mathbb{N}$ and $1\le p \le \infty$ we have
\eqnsl{
\norm{\pn f}_{W^{m,p}(\mathbb{T}^d)} \le C(n,m,d) \nd{\pn f }.
}{operatorPnProp1}
The obtained result is not surprising and could be justified based on the equivalence of norms in finite-dimensional spaces. 
Also, we can easily check that the order of differentiation and $\pn$, when sequentially applied to function $f:\td \rightarrow \mathbb{C}$, is interchangeable 
\eqns{
\n{\pn \frac{\partial f}{\partial x_i}}(x) 
& = \sum_{|k| < n} \int_{\mathbb{T}^d} \frac{\partial f(x^\prime)}{\partial x_i^\prime}  e^{-2 \pi i k x^\prime} dx^\prime e^{2 \pi i k x}
= \sum_{|k| < n} 2 \pi i k_i \int_{\mathbb{T}^d} f(x^\prime) e^{-2 \pi i k x^\prime} dx^\prime e^{2 \pi i k x} \\ &
= \frac{\partial }{\partial x_i} \n{\sum_{|k| < n} \int_{\mathbb{T}^d} f(x^\prime) e^{-2 \pi i k x^\prime} dx^\prime e^{2 \pi i k x} } 
= \n{ \frac{\partial }{\partial x_i} \pn f} (x).
}
Thus it is clear that the following hold:
\eqnsl{
\pn \divv f = \divv \pn f, ~~~~
\pn \nabla f = \nabla \pn f, ~~~~
\pn \Delta f = \Delta \pn f.
}{pnDerivativeCommutatorProp}
Before we define the ODE system we aim to solve (for $\alpha^n_{k,j}$, $\gamma^n_k$, $\eta^n_k$), we introduce the following functions
\eqnsl{
& v_{j}^n(x,t) = \sum_{|k|<n} \alpha^n_{k,j}(t) e^{2 \pi i k x },
~~~~
\om^n(x,t) = \sum_{|k|<n} \beta^n_k(t) e^{2 \pi i k x },
\\ &
b^n(x,t) = \sum_{|k|<n} \gamma^n_k(t) e^{2 \pi i k x },
~~~~~~
p^n(x,t) = \sum_{|k|<n} \eta^n_k(t) e^{2 \pi i k x }.
}{approxDef}
Thanks to orthonormality of basis $\{ e^{2 \pi i k x } \}_{k \in \mathbb{Z}^d}$ in $L^2(\td)$ we have 
\eqnsl{
\ndk{\vn(t)} = \sum_{|k|<n} |\alpha^n_{k,j}(t)|^2,~~~~
\ndk{\on(t)} = \sum_{|k|<n} |\beta^n_{k}(t)|^2,~~~~
\ndk{\bn(t)} = \sum_{|k|<n} |\gamma^n_{k}(t)|^2.
}{pressureEqA0}
Additionally, we define function $\overline{\nn}$ is the following way:
\eqnsl{
\overline{\nn} = \frac{\Pht (\re (\bn)) }{\Pst (\re(\on))}.
}{nnA1}
The modification was introduced to guarantee positive signs of diffusive terms. Moreover, the presence of $\re(\cdot)$ in the definition of $\overline{\nn}$ allows us to deal with possibly complex-valued solutions. 
Now we consider the following system of equations
\begin{align}
 \partial_t \alpha^n_{k,j}
& =
\left(
\n{
-\pn \left( \vn \cdot \nab \vn \right)
+ 
\divv \n{  \pn \left( \overline{\nn} D (\vn) \right) }
+
\nabla p_n
}_j
, e^{-2\pi i k \cdot} 
\right), 
\label{vnEqC1} \\ 
\partial_t \beta_k & = 
\n{
- \pn \left(\vn \cdot \nab \on \right) +  \divv \n{ \pn \left( \overline{\nn} \nab \on \right)} - \alpha \pn \left( \on^2 \right),
e^{-2\pi i k \cdot}
}, \label{omnEqC1} \\
\partial_t \gamma_k & =
\n{- \pn \left(\vn \cdot \nab \bn \right) + \divv \n{ \pn \left( \overline{\nn} \nab \bn \right) }
 - \pn \left( \bn  \on \right) + \pn \left( \overline{\nn} |D \vn |^2 \right)
, e^{-2\pi i k \cdot}
} \label{bnEqC1}
\end{align}
supplied with initial conditions 
\eqnsl{
\alpha_{k,j}^n(0) = \n{v_{0,j}, e^{-2\pi i k \cdot}}, ~~~~
\beta_{k}^n(0) = \n{\omega_{0}, e^{-2\pi i k \cdot}}, ~~~~
\gamma_{k}^n(0) = \n{b_{0}, e^{-2\pi i k \cdot}}
}{initEqC1}
and equation from which pressure is calculated
\eqnsl{
- \Delta p_n
=
\divv \Big[\pn \left( \vn \cdot \nab \vn \right) - \divv \n{  \pn \left( \overline{\nn} D (\vn) \right)} \Big], \quad \quad \quad
\int_{\mathbb{T}^d} p_n(x) dx = 0.
}{pressureEqA2}
The introduced system of equations can be represented in the following form
	\begin{align}
    \frac{d}{dt} & \begin{bmatrix}
           \n{\alpha_{k,j}^n}_{\substack{k = 1,...,n \\ j=1,...,d}} \\
           \n{\beta_{k}^n}_{k = 1,...,n} \\
           \n{\gamma_{k}^n}_{k = 1,...,n}
         \end{bmatrix}
         = F\n{ 
         \n{\alpha_{k,j}^n}_{\substack{k = 1,...,n \\ j=1,...,d}},
         \n{\beta_{k}^n}_{k = 1,...,n}, 
         \n{\gamma_{k}^n}_{k = 1,...,n}, 
         t 
         }.
    \end{align}
To show the existence of the solution of system (\ref{vnEqC1}) - (\ref{pressureEqA2}) we will show that the right-hand side is locally Lipschitz continuous with respect to $ \alpha_{k,j}^n $, $\beta_{k}^n$, $\gamma_{k}^n$, so basically, we need to estimate 
\eqns{
|
F(\alpha_{k,j}^{n,2},\beta_{k}^{n,2}, \gamma_{k}^{n,2}, t)
-
F(\alpha_{k,j}^{n,1},\beta_{k}^{n,1}, \gamma_{k}^{n,1}, t )
|.
} 
To this end, we introduce the following functions
\eqnsl{
v_{n,j}^{m}(x) = \sum_{|k|<n} \alpha^{n,m}_{k,j} e^{2 \pi i k x },
~~~~
\om^{m}_n(x) = \sum_{|k|<n} \beta^{n,m}_k e^{2 \pi i k x },
~~~~
b^{m}_n(x) = \sum_{|k|<n} \gamma^{n,m}_k e^{2 \pi i k x }.
}{pressureEqA3.05}
We also introduce 
\eqnsl{
\overline{\nn^m} = \frac{\Pht (\re (\bn^m)) }{\Pst (\re(\on^m))}.
}{pressureEqA3.1}
Additionally, functions $p_n^m$ are calculated with the help of the system (\ref{pressureEqA2}), with natural substitution of functions on the right-hand side: $v_{n,j} \rightarrow v_{n,j}^{m}$, $\om_{n} \rightarrow \om_{n}^{m}$, $b_{n} \rightarrow b_{n}^{m}$. \newline \noindent
We start checking local Lipschitz continuity by considering the term of the right-hand side of (\ref{vnEqC1}) involving pressure term. We see that
\eqns{
\left| 
\left( (\nabla p_n^2)_j , e^{-2\pi i k \cdot }  \right) 
-
\left( (\nabla p_n^1)_j , e^{-2\pi i k \cdot }  \right)
\right|
\le \nd{e^{-2\pi i k \cdot }} \nd{\nabla \n{p_n^2 - p_n^1}}
\le \nd{\nabla \n{p_n^2 - p_n^1}}.
}
From the basic theory of elliptic partial differential equations, the solution of (\ref{pressureEqA2}) system exists and the following estimate holds 
\eqnsl{
\ndk{\nabla \left( p_n^2 - p_n^1 \right)} 
& \le C \big(
\ndk{
\pn \left( \vn^2 \cdot \nab \vn^2 \right)
-
\pn \left( \vn^1 \cdot \nab \vn^1 \right)
}\\
& +
\ndk{ 
  \divv \n{ \pn  \left( \overline{\nn^2} D \vn^2 \right)}
- \divv \n{ \pn  \left( \overline{\nn^1} D \vn^1 \right)} 
}
\big).
}{pressureEqA4}
Let us analyse terms of the right-hand side of (\ref{pressureEqA4}) separately. First by using triangle inequality, H\"older inequality and (\ref{operatorPnProp1}) we get
\eqnsl{
& \nd{
\pn \left( \vn^2 \cdot \nab \vn^2 \right)
-
\pn \left( \vn^1 \cdot \nab \vn^1 \right)
}
\le 
\nd{
(\vn^2 - \vn^1) \cdot \nab \vn^2
-
\vn^1 \cdot \nab (\vn^1 - \vn^2)
}\\
& ~~~~\le 
\nif{\nab \vn^2}
\nd{
\vn^2 - \vn^1
}
+
\nif{\vn^1}
\nd{
\nab (\vn^1 - \vn^2)
} 
\le 
C (\nd{\vn^2} + \nd{\vn^1})
\nd{
\vn^2 - \vn^1
}.
}{pressureEqA5}
Now, we go back to analysing the second term of the right-hand side of (\ref{pressureEqA4}). Again, by using (\ref{operatorPnProp1}) we get 
\eqns{
\nd{ 
  \divv \n{ \pn  \left( \overline{\nn^2} D \vn^2 \right)}
- \divv \n{ \pn  \left( \overline{\nn^1} D \vn^1 \right)}
} 
\le 
C \nd{ 
  \overline{\nn^2} D \vn^2
- \overline{\nn^1} D \vn^1
}.
}
By using triangle inequality and H\"older inequality we get 
\eqns{
& \nd{ 
  \divv \n{ \pn  \left( \overline{\nn^2} D \vn^2 \right)}
- \divv \n{ \pn  \left( \overline{\nn^1} D \vn^1 \right)}
} \\
& ~~~~ \le C \nd{ 
  (\overline{\nn^2} - \overline{\nn^1}) D (\vn^2)
- \overline{\nn^1} D (\vn^1 - \vn^2)
} \\
& ~~~~ \le C \big(\nif{D (\vn^2)} \nd{ 
  \overline{\nn^2} - \overline{\nn^1}}
+ 
\nif{\overline{\nn^1}} \nd{D (\vn^1 - \vn^2)}
\big).
}
With the help of (\ref{operatorPnProp1}) we get 
\eqnsl{
& \nd{ 
  \divv \n{ \pn  \left( \overline{\nn^2} D \vn^2 \right)}
- \divv \n{ \pn  \left( \overline{\nn^1} D \vn^1 \right)}
} 
\\ & ~~~~~~ 
\le C \left(\nd{
\vn^2
}
\nd{ 
  \overline{\nn^2} - \overline{\nn^1}}
+ 
\nif{\overline{\nn^1}}
\nd{
\vn^2 - \vn^1
}
\right). 
}{pressureEqA6}
We see that based on (\ref{pressureEqA3.1}), (\ref{defPsi}), (\ref{defPhi}) and (\ref{operatorPnProp1}) we have
\eqnsl{
\nif{\overline{\nn^1}} 
=
\nif{\frac{\Pht \n{\re(\bn^1)} }{\Pst \n{\re(\on^1)}}}
\le 
\frac{\frac{1}{2}\bmt + \nif{\bn^1}}{\frac{1}{2} \omt }
\le 
\frac{\bmt + 2C\nd{\bn^1}}{\omt } 
}{pressureEqA7}
and
\eqns{
& \nd{ 
\overline{\nn^2} - \overline{\nn^1}
}
=
\nd{ 
\frac{\Pht (\re \bn^2) - \Pht (\re \bn^1) }{\Pst (\re \on^2)}
- 
\Pht (\re \bn^1) \frac{ \Pst (\re \on^2) - \Pst (\re \on^1)}{\Pst (\re \on^1)\Pst (\re \on^2)}
} \\
& \quad \quad \le
\frac{2}{\omt} 
\nd{ 
\Pht (\re \bn^2 ) - \Pht (\re \bn^1) }
+ \frac{\frac{1}{2}\bmt + C\nd{\bn^1}}{ \n{\frac{1}{2}\omt}^2 } \nd{\Pst (\re \on^2) - \Pst (\re \on^1)}
.
}
Now we use the fact that functions $\Pst$, $\Pht$ are Lipschitz continuous and obtain 
\eqnsl{
\nd{ 
\overline{\nn^2} - \overline{\nn^1}
}
\le C\left(
\frac{2}{\omt} 
\nd{ 
\bn^2 - \bn^1 }
+ \frac{\frac{1}{2}\bmt + C\nd{ \bn^1}}{ \n{\frac{1}{2}\omt}^2 } \nd{\on^2 - \on^1}
\right).
}{pressureEqA8}
Thus by plugging (\ref{pressureEqA8}) and (\ref{pressureEqA7}) into (\ref{pressureEqA6}) we get
\eqnsl{
& \nd{ 
  \divv \n{ \pn  \left( \overline{\nn^2} D (\pn \vn^2) \right)}
- \divv \n{ \pn  \left( \overline{\nn^1} D (\pn \vn^1) \right)}
} \\
& ~~ \le C_1 
\bigg(
\nd{\vn^2 - \vn^1}
+
\nd{\bn^2 - \bn^1}
+
\nd{\on^2 - \on^1}
\bigg),
}{pressureEqA9}
where $C_1 = C_1 \big(\nd{b^1_n}, \nd{\vn^2}, \bmt, \omt \big)$. Thus by using (\ref{pressureEqA9}), (\ref{pressureEqA5}), (\ref{pressureEqA4}) with the help of (\ref{pressureEqA3.05}), (\ref{pressureEqA0}) we get 
\eqns{
& \ndk{\nabla \left( p_n^2 - p_n^1 \right)}  
\le C_2 
\bigg(
\sum_{j=1}^d \sum_{|k|<n}
|\alpha^{n,2}_{k,j}-\alpha^{n,1}_{k,j}|^2
+
|\beta^{n,2}_{k}-\beta^{n,1}_{k,j}|^2
+
\sum_{|k|<n}
|\gamma^{n,2}_{k}-\gamma^{n,1}_{k}|^2
\bigg),
}
where 
\eqns{
C_2 = C_2 \left(
\sum_{|k|<n} |\gamma^{n,1}_{k}|^2, \sum_{j=1}^d \sum_{|k|<n}
|\alpha^{n,1}_{k,j}|^2, \sum_{j=1}^d \sum_{|k|<n}
|\alpha^{n,2}_{k,j}|^2, \bmt, \omt 
\right).
} 
Verification of the Lipschitz condition for other terms is analogous to conducted calculations. We will give one more estimate 
\eqns{
\big|
\n{ \pn \left( \overline{\nn^2} |D \vn^2 |^2 \right), e^{-2\pi i k \cdot}}
& -
\n{ \pn \left( \overline{\nn^1} |D \vn^1 |^2 \right), e^{-2\pi i k \cdot}}
\big| \\ & 
\le 
C \nd{
\pn \left( \overline{\nn^2} |D \vn^2 |^2 \right)
-
\pn \left( \overline{\nn^1} |D \vn^1 |^2 \right)
}.
}
By H\"older's inequality, we get
\eqns{
& \nd{
\pn \left( \overline{\nn^2} |D \vn^2 |^2 \right)
-
\pn \left( \overline{\nn^1} |D \vn^1 |^2 \right)
} \\
& ~~~\le
\nd{\overline{\nn^2} - \overline{\nn^1}} \nif{D \vn^2}^2 + \nif{\overline{\nn^1}}\nif{D \vn^2 + D \vn^1} \nd{D \vn^2 - D \vn^1}.
}
Now, based on (\ref{pressureEqA7}), (\ref{pressureEqA8}), (\ref{operatorPnProp1}) and (\ref{pressureEqA3.05}) we get
\eqns{
& \nd{
\pn \left( \overline{\nn^2} |D \pn \vn^2 |^2 \right)
-
\pn \left( \overline{\nn^1} |D \pn \vn^1 |^2 \right)
} \\ & ~~~~~~~~~~~~~~~~~~~~
\le \overset{\sim}{C} 
\bigg(
\nd{\vn^2 - \vn^1} 
+
\nd{\bn^2 - \bn^1} 
+
\nd{\on^2 - \on^1}
\bigg),
}
where 
\eqns{
\overset{\sim}{C} = \overset{\sim}{C} \left(
\sum_{|k|<n} |\gamma^{n,1}_{k}|^2, \sum_{j=1}^d \sum_{|k|<n}
|\alpha^{n,1}_{k,j}|^2, \sum_{j=1}^d \sum_{|k|<n}
|\alpha^{n,2}_{k,j}|^2, \bmt, \omt \right).
}
Because all right-hand side's terms in (\ref{vnEqC1}) - (\ref{bnEqC1}) are locally Lipschitz continuous, thus the existence of the unique solution for some $T_n> 0$ follows from Cauchy-Lipschitz theorem. Now let us multiply equations (\ref{vnEqC1})-(\ref{bnEqC1}) by $e^{2\pi i k x}$ and make summation over $|k|<n$: 
\eqnsl{
\partial_t \vn  + \pn \left( \vn \cdot \nab \vn \right) & -  \divv \n{ \pn \left( \overline{\nn} D \vn \right)}
+ \nabla p_n = 0,
}{vnEqA1}
\eqnsl{
\partial_t \on + \pn \left( \vn \cdot \nab \on \right) & -  \divv \n{ \pn \left( \overline{\nn} \nab  \on \right)} = - \alpha \pn \left( \on^2 \right) ,
}{onEqA1}
\eqnsl{
\partial_t \bn + \pn & \left( \vn \cdot \nab \bn \right)  - \divv \n{ \pn \left( \overline{\nn} \nab \bn \right) }
 = - \pn \left( \bn  \on  \right) + \pn \left( \overline{\nn} |D \vn |^2 \right),
}{bnEqA1}
\eqns{
\vn(0,x) = \pn v_0(x), ~~~~
\on(0,x) = \pn \om_0(x), ~~~~
\bn(0,x) = \pn b_0(x).
}
Now, by applying divergence to equation (\ref{vnEqA1}) and by using (\ref{pressureEqA2}) we get
\eqnsl{
\divv \vn  = 0.
}{divEqB1}
Also by taking imaginary parts of the system (\ref{vnEqA1}) - (\ref{bnEqA1}) and having in mind that initial data is real-valued is easy to check that solutions $(\vn, \on, \bn)$ are also real-valued. We will provide more details for $\vn$. Let us apply $\im $ to equation (\ref{vnEqA1}), multiply result by $\im \vn$ and integrate over $\td$ to obtain
\eqns{
\n{\partial_t \im \vn, \im \vn }  & 
+ \n{ \pn \left( \im \vn \cdot \nab \re \vn \right), \im \vn}
+ \n{ \pn \left( \re \vn \cdot \nab \im \vn \right), \im \vn }
\\ &
+ \n{ \pn \left( \overline{\nn} D \im \vn \right), D \im \vn}
- \n{ \im p_n , \divv (\im \vn) } = 0.
}
The last term on the left-hand side is zero due to (\ref{divEqB1}). Moreover, by  (\ref{nnA1}), (\ref{defPsi}), (\ref{defPhi}) we have
\eqns{
\jd \ddt \ndk{\im \vn}
+ \mnit \ndk{ D \im \vn }
& \le \nif{\nab \re \vn} \ndk{ \im \vn } 
\\ & \phantom{\le}
+ \nif{\re \vn} \nd{\im \vn } \nd{\nab \im \vn}.
}
By using Young inequality we get
\eqns{
\jd \ddt \ndk{\im \vn} 
+ \frac{\mnit}{2} \ndk{ \nabla \im \vn }
& \le \nif{\nab \re \vn} \ndk{ \im \vn }
+ \frac{1}{2 \mnit}\nif{\re \vn}^2 \nd{\im \vn }^2 \\ & + \frac{\mnit}{2} \nd{\nab \im \vn}^2.
}
Using Gr\"ownwall lemma and the fact that initial data is real-valued we conclude that
\eqns{
\ndk{\im \vn(t)} = 0 ~~~\forall t\in [0,T_n)
}
and thus that velocity is real-valued. A similar approach can be applied to $\on$ and $\bn$.
Thus, system (\ref{vnEqA1}) - (\ref{bnEqA1}) can be rewritten in a following way:
\eqnsl{
\divv \vn = 0,
}{divnEq}
\eqnsl{
\partial_t \vn + \pn \left(\vn \cdot \nab \vn \right) - \divv \n{ \pn  \left( \nn D \vn \right)} + \nabla p_n = 0,
}{vnEq}
\eqnsl{
\partial_t \on + \pn \left(\vn \cdot \nab \on \right) - \divv \n{ \pn  \left( \nn \nab \on \right)} = - \alpha \pn \left( \on^2 \right),
}{onEq}
\eqnsl{
\partial_t \bn + \pn \left(\vn \cdot \nab \bn \right) -\divv \n{ \pn  \left( \nn \nab \bn \right)} = - \pn \left(\bn  \on \right)+ \pn \left( \nn |D \vn |^2 \right),
}{bnEq}
\eqnsl{
\vn(0,x) = \pn v_0(x), ~~~~
\on(0,x) = \pn \om_0(x), ~~~~
\bn(0,x) = \pn b_0(x),
}{initn}
where:
\eqnsl{
\nn = \frac{\Pht (\bn) }{\Pst (\on)}.
}{defMn}
From the equation (\ref{approxDef}) it is clear that functions $(\vn, \on, \bn)$ are smooth with respect to spatial coordinates. By employing a standard iterative approach from the theory of ODE's ($C^k$ right-hand side implies $C^{k+1}$ solution) applied to the system (\ref{vnEqC1})-(\ref{bnEqC1}), it is easy to conclude that $(\vn, \on, \bn)$ are also smooth with respect to time.
\subsection{Energy estimates}
Before we start deriving energy estimates we need to establish some relations between $\ls$ and the derivative. Let $f:\td \rightarrow \mathbb{C}$, $w:\td \rightarrow \mathbb{C}^d$ such that $\pn f = f$ and $\pn w = w$. Let us recall that $\hat{f}(k) = \int_{\td} f(x) e^{-2 \pi i x \cdot k} dx$. Then we have
\eqns{
\frac{\partial \ls f}{\partial x_i}(x)
& = 
\frac{\partial}{\partial x_i}
\sum_{k \in \mathbb{Z}^d: |k|<n} \left( 1 + 4 \pi^2 |k|^2 \right)^{s/2} e^{2 \pi i x \cdot k} \hat{f}(k) \\
&  = 
\sum_{k \in \mathbb{Z}^d: |k|<n} \left( 1 + 4 \pi^2 |k|^2 \right)^{s/2} e^{2 \pi i x \cdot k} (2 \pi i k_i) \hat{f}(k).
}
Now using the properties of Fourier transform acting on a derivative we get 
\eqns{
&\frac{\partial \ls f}{\partial x_i}(x)
 = 
\sum_{k \in \mathbb{Z}^d: |k|<n} \left( 1 + 4 \pi^2 |k|^2 \right)^{s/2} e^{2 \pi i x \cdot k} \widehat{\n{\frac{\partial f}{\partial x_i}}}(k)
 = \ls \n{\frac{\partial f}{\partial x_i}}(x).
}
Thus we have
\eqnsl{
\nabla \ls f = \ls \nabla f, 
~~~~~~
D \ls f = \ls D f, 
~~~~~~
\ls \divv w = \divv \ls w, 
~~~~~~ 
\Delta \ls f = \ls \Delta f.
}{EstProp0}
Also, it can similarly be shown that
\eqnsl{
\pn \ls f = \ls \pn f.
}{EstProp0.5}
In the next part, we will be calculating the various inner products in $L^2(\td)$. Thus to simplify reasoning it is beneficial to observe that if function $f$ is a real-valued function so is $\ls f$. Before we proceed with energy estimates we need to introduce notation regarding constants dependent on time. Positive constants dependent on the time of the form $C(\eta_1, ...,\eta_m, t)$ used in the later parts of the proof can be expressed by
\eqnsl{
C(\eta_1, ...,\eta_m, t) = \widetilde{C}(\eta_1, ...,\eta_m) (1 + t)^\gamma
}{timeDepConstant}
for some $\gamma \in \mathbb{R}$.
Let us apply the $\ls$ operator to equation (\ref{vnEq}), multiply the result by $\ls \vn$, and integrate over $\td$. From this we obtain
\eqns{
\jd \ddt \nd{\ls \vn}^2  + \n{\ls \pn \n{ \vn \cdot \nab \vn}, \ls \vn} 
& - \n{\ls \divv \pn \n{\nn D \vn}, \ls \vn} \\
& = - \n{\ls \nabla p_n , \ls \vn}.
}
Using properties (\ref{EstProp0}), (\ref{EstProp0.5}), integration by parts, and the fact that $\divv \vn = 0$ we get 
\eqnsl{
\jd \ddt \nhs{\vn}^2 + \n{\ls \n{\nn D \vn}, \ls \nab \vn}
=
- \n{\ls \n{ \vn \cdot \nab \vn}, \ls \vn}.
}{EstVn0}
Now using H\"older inequality and Lemma \ref{Lem1.5} we get 
\eqnsl{
\left| \n{\ls \n{ \vn \cdot \nab \vn}, \ls \vn} \right| 
\le 
C \nhs{\vn}^2\nhs{\nab \vn}.
}{EstVn1} 
Using properties (\ref{EstProp0}) and integration by parts we get 
\eqns{
-\n{\ls \divv \n{ \nn D \vn}, \ls \vn}
=
\n{\ls \n{\nn D \vn},\nab \n{ \ls \vn}} .
}
Now, we rewrite the expression in a way that will enable us to use the commutator estimate:
\eqnsl{
\inld{\ls \left( \nn D \vn \right)}{\nab \ls \vn}
& =
\inld{\nn D \ls \vn}{\nab \ls \vn} 
+ \inld{ \left[ \ls, \nn \right] D \vn }{\nabla \ls \vn},  
}{EstVn1.1}
where $\left[ \ls, \nn \right] D \vn := \ls(\nn D \vn) - \nn \ls D \vn$. We want to estimate the above expression from below. Using H\"older inequality,  Lemmas \ref{Lem7}, \ref{hsNormOfGradient}, we get 
\eqnsl{
\inld{ \left[ \ls , \nn \right] D  \vn}{\nabla \ls \vn}
& \le 
C \left( \nif{\nabla \nn} \nhsmj{\nab \vn} + \nhs{\nn}\nif{\nab \vn} \right)
\nhs{\nab \vn} \\
& \le 
C \left( \nif{\nabla \nn} \nhs{\vn} + \nhs{\nn}\nif{\nab \vn} \right)
\nhs{\nab \vn}.
}{EstVn1.2}
For now, we will leave inequality in the above form. By definitions (\ref{h}), (\ref{defPsi}), (\ref{defPhi}) and (\ref{defMn})  we have  
\eqnsl{
\inld{\nn D \ls \vn}{\nab \ls \vn}
=
\inld{\nn D \ls \vn}{D \ls \vn} 
\geq 
\mnit \nhs{D \vn}^2.
}{EstVn1.99}
Let us rewrite the right-hand side in the following way 
\eqns{
\nhs{D \vn}^2
& = 
  \sum_{i,j=1}^d \n{ \frac{\n{\ls \vn}_{i,j} + \n{\ls \vn}_{j,i}}{2}, \frac{\n{\ls \vn}_{i,j} + \n{\ls \vn}_{j,i}}{2}} 
\\ & = 
\frac{1}{2} \sum_{i,j=1}^d 
\n{ \n{\ls \vn}_{i,j}, \n{\ls \vn}_{i,j}}
+
\n{ \n{\ls \vn}_{i,j}, \n{\ls \vn}_{j,i}}.
}
By performing integration by parts and using the fact that $\divv \vn = 0 $ in combination with (\ref{EstProp0}) and (\ref{EstVn1.99}) we get 
\eqns{
\inld{\nn D \ls \vn}{\nab \ls \vn}
\geq 
\jd \mnit \nhs{ \nab \vn}^2.
}
By combining (\ref{EstVn1.1}), (\ref{EstVn1.2}) and above inequality, we get 
\eqnsl{
\inhs{\nn D \vn}{\nab \vn} 
& \geq
\mnit \nhs{\nab \vn}^2
 - C \nif{\nabla \nn} \nhs{\vn} \nhs{\nab \vn} \\
& \phantom{\le} - C \nhs{\nn}\nif{\nab \vn} \nhs{\nab \vn}.
}{EstVn2}
Using estimates (\ref{EstVn1}) and (\ref{EstVn2}) in (\ref{EstVn0}) we get 
\eqns{
\jd & \ddt \nhs{\vn}^2 + \mnit \nhs{\nab \vn}^2 \\
& \le
C \Big( 
\nif{\nabla \nn} \nhs{\vn} \nhs{\nab \vn}
+ 
\nhs{\nn}\nif{\nab \vn} \nhs{\nab \vn} 
 +
\nhs{\vn}^2 \nhs{\nab \vn}
\Big).
}
Now, according to Lemma \ref{hsNormOfGradient} we can express $\nhs{\nab \vn}^2$ in the following way
\eqns{
\nhs{\nab \vn}^2 = \nhsj{\vn}^2 - \nhs{\vn}^2.
}
We can rewrite inequality in the following way:
\eqnsl{
\jd & \ddt \nhs{\vn}^2 + \mnit \nhsj{\vn}^2 \le  \mnit \nhs{\vn}^2 \\
& +
C \Big( 
\nif{\nabla \nn} \nhs{\vn} \nhsj{\vn}
+ 
\nhs{\nn}\nif{\nab \vn} \nhsj{\vn} 
+
\nhs{\vn}^2 \nhsj{\vn}
\Big).
}{EstVn3}
Now, we will proceed with acquiring the estimate on $\on$. Let us apply $\ls$ operator to the equation (\ref{onEq}) and take inner product with with $\ls \on$:
\eqnsl{
\jd \ddt \nhs{\on}^2 + \inhs{\nn \nab \on}{\nab \on} = - \inhs{\vn \cdot \nab \on}{\on} - \alpha \inhs{\on^2}{\on}.
}{EstOn0}
Proceeding as in (\ref{EstVn2}), we obtain   
\eqns{
\inhs{\nn \nab \on}{\nab \on}
=
\n{ \nn \ls \nab  \on,  \ls \nab \on}
+ 
\n{ \ls \n{\nn  \nab  \on} - \nn \ls \n{ \nab \on}, \ls \nab \on}.
}
Thus using H\"older inequality and Lemma \ref{Lem7} we get 
\eqnsl{
\inhs{\nn \nab \on}{\nab \on} 
& \geq
\mnit \nhs{\nab \on}^2 - 
C \nif{\nabla \nn} \nhs{\on} \nhs{\nab \on} \\
& 
\phantom{\le} - C \nhs{\nn}\nif{\nab \on} \nhs{\nab \on}.
}{EstOn1}
Now, using H\"older inequality in combination with Lemma \ref{Lem1.5}, we treat remaining nonlinearities in a following way: 
\begin{align}
\left |\inhs{\vn \cdot \nab \on}{\on} \right| 
& \le \nhs{\vn}\nhs{\on}\nhs{\nab \on},
\label{EstOn2}
\\
\left |\inhs{\on^2 }{\on} \right| 
& \le 
C \nhs{\on}^3.
\label{EstOn3}
\end{align}
Applying inequalities (\ref{EstOn1}), (\ref{EstOn2}), and (\ref{EstOn3}) to (\ref{EstOn0}) we get 
\eqns{
\jd \ddt \nhs{\on}^2 & + \mnit \nhs{\nab \on}^2 
\le
C \Big( 
\nif{\nabla \nn} \nhs{\on} \nhs{\nab \on} \\
& + 
\nhs{\nn}\nif{\nab \on} \nhs{\nab \on}
+
\nhs{\on}^3 + 
\nhs{\vn}\nhs{\on}\nhs{\nab \on}
\Big).
}
Using Lemma \ref{hsNormOfGradient}, we can rewrite it in a more suitable form
\eqnsl{
\jd \ddt \nhs{\on}^2 & + \mnit \nhsj{\on}^2 
\le
\mnit \nhs{\on}^2 + C \Big( 
\nif{\nabla \nn} \nhs{\on} \nhsj{\on} \\
& + 
\nhs{\nn}\nif{\nab \on} \nhsj{\on}
 +
\nhs{\on}^3 + 
\nhs{\vn}\nhs{\on}\nhsj{\on}
\Big).
}{EstOn4}  
Now estimates for $\bn$ will be provided. As before, let us apply $\ls$ operator to the equation (\ref{bnEq}) and take inner product with with $\ls \bn$:
\eqnsl{
\jd \ddt \nhs{\bn}^2 + \inhs{\nn \nab \bn}{\nab \bn} 
& = 
- \inhs{\vn \cdot \nab \bn}{\bn} 
- \inhs{\bn \on}{\bn} 
\\ & \phantom{=}
+ \inhs{\nn \dvnk}{\bn} .
}{EstBn0}
Proceeding as before, with the use of Lemmas \ref{Lem7} and \ref{Lem1.5} we get:
\begin{align}
\left |\inhs{\vn \cdot \nab \bn}{\bn} \right| 
& \le \nhs{\vn}\nhs{\bn}\nhs{\nab \bn},
\label{EstBn1} \\
\left |\inhs{\bn \on }{\bn} \right| 
& \le 
C \nhs{\bn}^2 \nhs{\on},
\label{EstBn2} \\
\inhs{\nn \nab \bn}{\nab \bn} 
& \geq
\mnit \nhs{\nab \bn}^2
-  
C 
\nif{\nabla \nn} \nhs{\bn} \nhs{\nab \bn}  
\label{EstBn3}
\\ & \phantom{\geq} - C\nhs{\nn}\nif{\nab \bn} \nhs{\nab \bn}. \nonumber
\end{align}
Now we provide the estimate for the last term of r.h.s. of (\ref{EstBn0}). Using Lemmas \ref{Lem1.5} and \ref{Lem1} we get
\eqnsl{
\left| \inhs{\nn \dvnk}{\bn} \right| 
& \le  \nhs{\nn} \nhs{\dvnk} \nhs{\bn} \\
& \le C \nhs{\nn} \nif{\nab \vn} \nhs{\nab \vn} \nhs{\bn}.
}{EstBn4}
Finally, by using estimates (\ref{EstBn1}), (\ref{EstBn2}), (\ref{EstBn3}), (\ref{EstBn4}) in (\ref{EstBn0}) and applying Lemma \ref{hsNormOfGradient} we get 
\eqnsl{
\jd & \ddt \nhs{\bn}^2 + \mnit \nhsj{\bn}^2 
\le
\mnit \nhs{\bn}^2 + C \bigg(
\nif{\nabla \nn} \nhs{\bn} \nhsj{\bn} 
\\ & 
+ \nhs{\nn}\nif{\nab \bn} \nhsj{\bn} 
+ \nhs{\bn}^2 \nhs{\on} 
+ \nhs{\vn}\nhs{\bn}\nhsj{\bn} 
\\ & 
+ \nhs{\nn} \nif{\nab \vn} \nhs{\nab \vn} \nhs{\bn}
\bigg).
}{EstBn7}
By summing inequalities (\ref{EstVn3}), (\ref{EstOn4}) and (\ref{EstBn7}) we get:
\eqnsl{
\jd & \ddt \nhs{\vn, \on, \bn}^2 + \mnit \nhsj{\vn, \on, \bn}^2 
\le
\mnit \nhs{\vn, \on, \bn}^2
\\&
 + C \bigg(
\nhs{\nn} \bk{\nif{\nab \vn} + \nif{\nab \on} + \nif{\nab \bn}} \nhsj{\vn, \on, \bn} 
+ \nhs{\vn, \on, \bn}^3
\\ &  
+\nif{\nab \nn} \nhs{\vn, \on, \bn} \nhsj{\vn, \on, \bn}
+\nhs{\vn, \on, \bn}^2 \nhsj{\vn, \on, \bn}
\\ & 
+ \nhs{\nn} \nif{\nab \vn} \nhs{\nab \vn} \nhs{\bn} \bigg),
}{EstSum0}
where $\nhs{\vn, \on, \bn}^2 = \nhs{\vn}^2 + \nhs{\on}^2 +\nhs{\bn}^2$. Now let us see that by definition (\ref{defMn}) and (\ref{defPsi})-(\ref{estiPhi}) we have 
\eqns{
\nif{\nab \nn} 
& = 
\nif{\nab \frac{\Pht (\bn) }{\Pst (\on)} }  
=
\nif{\frac{\Pht^{'} (\bn) }{\Pst (\on)} \nab \bn - \frac{\Pht (\bn) }{\Pst^2 (\on)} \Pst^{'} (\on) \nab \on } \\
& \le C(\omi, t) 
\bk{
\nif{\nab \bn} 
+ 
\nif{\Pht (\bn) } \nif{\nab \on}
} \\
& \le C(\omi, t) 
\bk{
\nif{\nab \bn} 
+ 
\bk{\bmi + \nif{\bn}} \nif{\nab \on},
}
}
where constant $C(\omi, t)$ is as in (\ref{timeDepConstant}).
By using Lemma \ref{Lem6} we get
\eqns{
\nif{\nab \nn} \le C(\omi, t) 
\bk{
\nif{\nab \bn} 
+ 
\bk{\bmi + \nhs{\bn}} \nif{\nab \on}
}.
}
To simplify expressions we will introduce polynomial notation: for $k \in \mathbb{R}_+$ we define $P_k(t)$ in a following way:
\eqnsl{
P_k(t) = \left(
1 + \nhs{\vn(t)}^2 + \nhs{\on(t)}^2 + \nhs{\bn(t)}^2
\right)^{k/2}.
}{defPk}
Thus we can write 
\eqnsl{
\nif{\nab \nn} 
& \le C(\omi, \bmi, t) P_1(t)
\bk{
\nif{\nab \bn} + \nif{\nab \on}
}.
}{pom__1}
Now using (\ref{pom__1}) and (\ref{defPk}) in (\ref{EstSum0}) we get
\eqnsl{
\jd \ddt & \nhs{\vn, \on, \bn}^2 + \mnit \nhsj{\vn, \on, \bn}^2 
\le
\mnit P_2(t) \\
& + C \bigg(
\nhs{\nn} \bk{\nif{\nab \vn} + \nif{\nab \on} + \nif{\nab \bn}} \nhsj{\vn, \on, \bn}
+ P_3(t) 
\\ & 
+ P_2(t) \bk{\nif{\nab \bn} + \nif{\nab \on}} \nhsj{\vn, \on, \bn} 
+ P_2(t) \nhsj{\vn, \on, \bn}
\\ & 
+ \nhs{\nn} \nif{\nab \vn} P_1(t) \nhsj{\vn}
\bigg).
}{EstSum0.95}
By using properties of $P_k$ we can write
\eqnsl{
& \hspace{-2mm}\jd \ddt \nhs{\vn, \on, \bn}^2 + \mnit \nhsj{\vn, \on, \bn}^2 
\le \mnit P_2(t) + C \bigg(P_2(t) \nhsj{\vn, \on, \bn}
\\ & \hspace{-0.5mm} + P_3(t) + 
\n{P_2(t) + \nhs{\nn}P_1(t) } \bk{\nif{\nab \vn} + \nif{\nab \on} + \nif{\nab \bn}} \nhsj{\vn, \on, \bn}
\bigg).
}{EstSum1}
Now, we will continue the proof with the assumption that $s \in \left( \frac{d}{2}, \frac{d}{2} + 1\right]$ - because in other cases we have $\nm{\nab f}{\infty} \le C \nhs{f}$ and subsequent estimates simplify. Thus by using Lemma \ref{Lem9} we get:
\eqnsl{
\hspace{-2mm}\jd & \ddt \nhs{\vn, \on, \bn}^2 + \mnit \nhsj{\vn, \on, \bn}^2 
\le \mnit P_2(t) + C \Big(P_2(t) \nhsj{\vn, \on, \bn}
\\ \hspace{-2mm} & + P_3(t) + \n{P_2(t) + \nhs{\nn} P_1(t)}
P_{\jd \n{s - \frac{d}{2}}}(t) \nhsj{\vn, \on, \bn}^{1 - \jd \n{s - \frac{d}{2}}}   
\nhsj{\vn, \on, \bn}
\Big).
}{EstSum2}
In the above inequality term $\nhs{\nn}$ remains not estimated. First, let us consider three auxiliary estimates that follow from Lemma \ref{Lem2}, (\ref{estiPsi}) and (\ref{defPsi})
\eqnsl{
\nhs{\Pht(\bn)} & \le \nhs{\Pht(\bn) - \jd \bmt} + \nhs{\jd \bmt} \\ 
& \le C \ncs{\Pht^{'}} \left( 1 + \nhs{\bn} \right)^{\lceil s \rceil} \nhs{\bn} + \nhs{\jd \bmt} \\
& \le C \left( 1 + \nhs{\bn} \right)^{\lceil s \rceil} \nhs{\bn} + C,
}{EstMu1}
where $C = C(s, \bmi, t) $ is rational function dependent on time $t$, finite $\forall t\geq 0$ (see (\ref{timeDepConstant})). Similarly from Lemma \ref{Lem2}, (\ref{estiPhi}) and (\ref{defPhi}) we get 
\eqnsl{
\nhs{\frac{1}{\Pst(\on)}} 
& \le 
\nhs{\frac{1}{\Pst(\on)} - \frac{1}{\jd \omt}} + \nhs{\frac{1}{\jd \omt}} \\ 
& \le C \ncs{\left(\frac{1}{\Pst}\right)^{'}} \left( 1 + \nhs{\on} \right)^{\lceil s \rceil} \nhs{\on} + \nhs{\frac{2}{\omt}} \\
& \le C \left( 1 + \nhs{\on} \right)^{\lceil s \rceil} \nhs{\on} + C,
}{EstMu3}
where $C = C(s, \omi, \oma, t) $ is as in (\ref{timeDepConstant}). Now using obtained estimates in combination with Lemmas \ref{Lem1}, \ref{Lem6} and definitions (\ref{defPhi}), (\ref{defMn}) we can proceed with estimates on $\nhs{\nn}$ in the following way:
\eqnsl{
\nhs{\nn} 
& = \nhs{\frac{\Pht (\bn) }{\Pst (\on)}}
\le \nhs{\Pht(\bn)} \nif{\frac{1}{\Pst (\on)}} + \nif{\Pht(\bn)} \nhs{\frac{1}{\Pst (\on)}} \\
& \le \frac{2}{\omt} \nhs{\Pht(\bn)} + \nhs{\Pht(\bn)} \nhs{\frac{1}{\Pst (\on)}} \\
& \le C_1 
\left( \left( 1 + \nhs{\bn} \right)^{\lceil s \rceil} \nhs{\bn} + 1 \right)
\cdot  \left(\left( 1 + \nhs{\on} \right)^{\lceil s \rceil} \nhs{\on} + 1\right) \\
& \le C_1 P_{2\lceil s \rceil + 2}(t),
}{EstMu4}
where $C_1 = C_1(\omi, \oma, \bmi, t)$ is as in (\ref{timeDepConstant}). 
Now let us use estimate (\ref{EstMu4}) in (\ref{EstSum2}):
\eqns{
 \jd \ddt \nhs{\vn, \on, \bn}^2 & + \mnit \nhsj{\vn, \on, \bn}^2 
\le
\mnit P_2(t) + C \Big( 
P_2(t) \nhsj{\vn, \on, \bn}
\\ &  
+ P_3(t) + \n{P_2(t)  + P_{2\lceil s \rceil + 3}(t)} P_{\jd \n{s - \frac{d}{2}}}(t) \nhsj{\vn, \on, \bn}^{2 - \jd \n{s - \frac{d}{2}}} 
\Big).
}
Using the fact that for $k_1 \geq k_2$ we have $P_{k_1} \geq P_{k_2}$ and Young's inequality we get
\eqnsl{
\jd & \ddt \nhs{\vn, \on, \bn}^2 + \mnit \nhsj{\vn, \on, \bn}^2 
\le
C(\bmi, \omi, \oma, t) 
\\ &  \cdot \bigg( 
P_{2\lceil s \rceil + 3 + \jd \n{s - \frac{d}{2}}}(t)
\nhsj{\vn, \on, \bn}^{2 - \jd \n{s - \frac{d}{2}}} 
+ P_3(t)
+ P_2(t) \nhsj{\vn, \on, \bn}
 \bigg).
}{EstSum3}
Now, let us apply Young's inequality to the right-hand side with coefficients
\eqns{ 
\bk{\frac{4}{s - \frac{d}{2}}, \frac{2}{2 - \jd \bk{s - \frac{d}{2}}}}, 
~~~~
\bk{2, 2}
} to obtain
\eqns{
\ddt \nhs{\vn, \on, \bn}^2 + \mnit \nhsj{\vn, \on, \bn}^2
& \le
C(\bmi, \omi, \oma, t) 
 \\ & \cdot 
\bigg(
  P_{\bk{2\lceil s \rceil + 3 + \jd \n{s - \frac{d}{2}}} \frac{4}{s - \frac{d}{2}}}(t) 
+ P_4(t) + P_3(t) \bigg).
}
Now, let us introduce $\beta(s)>1$ such that
\eqns{ 
2\beta(s) = \max \left\{ 
 4,
 \bk{2\lceil s \rceil + 3 + \jd \n{s - \frac{d}{2}}} \frac{4}{s - \frac{d}{2}}
\right\}.
} Thus we have
\eqnsl{
\ddt & \nhs{\vn, \on, \bn}^2 + \mnit \nhsj{\vn, \on, \bn}^2
\le 
C(\bmi, \omi, \oma, s, t) P_{2\beta}(t).
}{EstSum5}
Thus by definition (\ref{defPk}) we get
\eqnsl{
\ddt \left( 1 + \nhs{\vn, \on, \bn}^2 \right) & + \mnit \nhsj{\vn, \on, \bn}^2 \\
& \le C(\bmi, \omi, s, t)
  \left( 1 + \nhs{\vn, \on, \bn}^2 \right)^\beta.
}{EstSum6}
By integrating from $0$ to $t$ we get inequality 
\eqns{
\frac{1}{-\beta + 1}\left( 1 + \nhs{\vn, \on, \bn}^2 \right)^{-\beta + 1}
& \le 
\frac{1}{-\beta + 1}\left( 1 + \nhs{\pn v_0, \pn \om_0, \pn b_0}^2 \right)^{-\beta + 1} \\
&  + \int_0^t C(\bmi, \omi, \oma, s, \tau) d \tau.
}
After some manipulations, we obtain a uniform estimate for approximated solutions
\eqnsl{
\hspace*{-2.4mm} 
\nhs{\vn, \on, \bn}^2 &
\le \frac{1}{ \n{
\left( 1 + \nhs{v_0, \om_0, b_0}^2 \right)^{1-\beta}
- (\beta - 1) \int_0^t C(\bmi, \omi, \oma,  s, \tau) d \tau}^\frac{1}{\beta - 1}} \\ & \phantom{\le}
 -1
}{EstSum7}
provided the denominator on the right-hand side is positive. Thus, let us define existence time $T$ such that the following equality holds 
\eqnsl{
(1 - 2^{-\beta + 1}) \left( 1 + \nhs{v_0, \om_0, b_0}^2 \right)^{-\beta + 1}
 = (\beta - 1) \int_0^T C(\bmi, \omi, \oma, s, \tau) d \tau. 
}{EstSum8}
Using (\ref{EstSum8}) in (\ref{EstSum7}) we derive the estimate:
\eqnsl{
\nhs{\vn, \on, \bn}^2
\le 2 \nhs{v_0, \om_0, b_0}^2 + 1 ~~~~\forall t \in [0,T].
}{EstSumF1} 
Additionally (\ref{EstSum6}) and (\ref{EstSumF1}) imply that:
\eqnsl{
\int_0^T \nhsj{\vn, \om, \bn }^2 d \tau \le C \left(\bmi, \omi, \oma, \nhs{v_0, \om_0, b_0}, s, T \right) < \infty.
}{EstSumF2}
To show the continuity of the solution, the estimate in the norm $L^2(0,T,H^{s-1}(\td) )$ for the time derivative of the solution is required. 
We will derive the estimate only for $\bn$ as the calculations for other variables are similar. Let us apply $J^{s-1}$ to equation (\ref{bnEq}) and calculate inner product with $\partial_t J^{s-1} \bn$
\eqns{
\nhsmj{\partial_t \bn}^2 
& =
- \inh{\vn \nab \bn}{\partial_t \bn }{s-1}
+ \inh{\nab \cdot \left(\nn \nab \bn \right)}{\partial_t \bn }{s-1}
 -\inh{\bn \on }{\partial_t \bn}{s-1}
\\ & \phantom{=}
+ \inh{\nn \dvnk }{\partial_t \bn}{s-1}.
}
Using H\"older and Young inequality we get 
\eqns{
\nhsmj{\partial_t \bn}^2 
& \le C \bigg(
  \nhsmj{\vn \nab \bn}^2
+ \nhsmj{\nab \cdot \left(\nn \nab \bn \right)}^2
+ \nhsmj{\bn \on }^2 
\\ & \phantom{\le}
+ \nhsmj{\nn \dvnk }^2
\bigg).
} 
By Lemmas \ref{hsNormOfGradient} and \ref{Lem10} it easily follows
\eqnsl{
\nhsmj{\partial_t \bn}^2 
& \le C \bigg(
  \nhs{\vn \nab \bn}^2
+ \nhs{\nn \nab \bn}^2
+ \nhs{\bn \on }^2 
+ \nhsmj{\nn \dvnk }^2
\bigg).
}{bt_estim_eq1}
The most troublesome term to estimate is the last one on the right-hand side. Let us concentrate on it first. Let us chose $\varepsilon \geq 0$ in the following way
\begin{equation} 
\left\{
\begin{array}{lr}
\varepsilon \in (0, \min \{ s - \frac{d}{2}, 1 \}) & \text{for } d = 2\\
\varepsilon = 0 & \text{for } d \geq 3
\end{array}  
\right. .
\end{equation}
Based on Lemma \ref{Lem1} we have
\eqns{
\hspace{-2mm}
\nhsmj{\nn \dvnk } \le C \left( 
\norm{J^{s-1} \nn}_{\frac{d}{\varepsilon + \frac{d}{2} - 1}}  
\norm{|D \vn|^2}_{\frac{d}{1 - \varepsilon}}    
+
\norm{\nn}_{\infty}  
\norm{J^{s-1} (|D \vn|^2)}_{2}  
\right).
}
By applying Lemma \ref{Lem1} to the last term on the right-hand side we get
\eqns{
\hspace{-2mm}
\nhsmj{\nn \dvnk } \le C \left( 
\norm{J^{s-1} \nn}_{\frac{d}{\varepsilon + \frac{d}{2} - 1}}  
\norm{D \vn}_{\frac{d}{1 - \varepsilon}}   
+
\norm{\nn}_{\infty}  
\norm{J^{s-1} D \vn}_{2}  
\right)
\norm{D \vn}_{\infty} .
}
Let us observe that based on Lemma \ref{Lem11} we have
\begin{align}
\norm{J^{s-1} \nn}_{\frac{d}{\varepsilon + \frac{d}{2} - 1}} 
& \le C
\norm{J^{s-\varepsilon} \nn}_{2},
\\
\norm{D \vn}_{\frac{d}{1 - \varepsilon}}
& \le C
\norm{J^{\frac{d}{2}+\varepsilon} \vn}_{2}.
\end{align}
Using the above estimates and Lemmas \ref{Lem6}, \ref{hsNormOfGradient} we get 
\eqns{
\nhsmj{\nn \dvnk } \le C \left( 
\norm{J^{s-\varepsilon} \nn}_{2}  
\norm{J^{\frac{d}{2} + \varepsilon} \vn}_{2}  
+
\nhs{\nn}  
\nhs{\vn}
\right) \nhsj{\vn}  .
}
Using Lemma \ref{Lem10} we get
\eqns{
\nhsmj{\nn \dvnk } 
\le C
\nhs{\nn}  
\nhs{\vn}  
\nhsj{\vn} .
}
Using the above result in (\ref{bt_estim_eq1}) and Lemmas \ref{Lem1.5}, \ref{hsNormOfGradient} we get  
\eqns{
\nhsmj{\partial_t \bn}^2
& \le C \bigg(
  \nhs{\vn}^2 \nhsj{\bn}^2
+ \nhs{ \nn }^2  \nhsj{ \bn }^2
+ \nhs{ \bn }^2  \nhs{ \on }^2 
\\ & \phantom{\le}
+ \nhs{\nn}^2  
\nhs{\vn}^2  
\nhsj{\vn}^2
\bigg).
}  
Right hand side of above expression is in $L^1(0,T)$ (due to (\ref{EstSumF1}), (\ref{EstSumF2}) and (\ref{EstMu4})) thus the estimate was proven. To conclude time derivative estimates are as follows
\eqnsl{
\norm{\partial_t \vn, \partial_t \on, \partial_t \bn}_{L^2(0,T,H^{s-1}(\td) )}
\le C \left(\bmi, \omi, \oma, \nhs{v_0, \om_0, b_0}, s, T \right) < \infty.
}{EstDt}
\subsection{Passage to the limit in approximate system, regularity of solution, uniqueness}
Using estimates (\ref{EstSumF1}), (\ref{EstSumF2}) and (\ref{EstDt}) we can conclude existence of sub-sequence $\{n_k \}$ (relabelled as $n$) such that:
\begin{align}
\vn & \rightarrow v & & \text{~~~~weakly in } L^2(0,T,H^{s+1}_{\divv}(\td)) ,
\label{ConvFunc1} \\
\om, \bn & \rightarrow \om , b & & \text{~~~~weakly in } L^2(0,T,H^{s+1}(\td)) , 
\label{ConvFunc2} \\
\partial_t \vn, \partial_t \om, \partial_t \bn 
& \rightarrow 
\partial_t v, \partial_t \om , \partial_t b
& & \text{~~~~weakly  in } L^2(0,T,H^{s-1}(\td)) , \label{ConvDT} \\
\vn, \om, \bn & \rightarrow v, \om , b & & \text{~~~~weakly$*$ in } L^\infty(0,T,H^{s}(\td)).
\label{ConvFunc3} 
\end{align}
Additionally, from the Aubin-Lions lemma, it follows that
\begin{align}
\vn, \om, \bn & \rightarrow v, \om , b & & \text{strongly in } L^2(0,T,H^{s^{\prime}+1}(\td)),
\label{AL1} \\
\vn, \om, \bn & \rightarrow v, \om , b & & \text{strongly in } C(0,T,H^{s^{\prime}}(\td))
\label{AL2}.
\end{align}
for all $s^{\prime} < s$. It is easy to see that 
\eqnsl{
\nn & \rightarrow \mu = \frac{\Pht (b) }{\Pst (\om)} \text{\quad \quad \quad strongly in } C(0,T,H^{s^{\prime}}(\td))
}{convMun}
holds for all $\frac{d}{2} < s^{\prime} < s$. Indeed, we see that with the help of triangle inequality and Lemma \ref{Lem1.5} we get
\eqns{
\nhsprim{\nn - \mu} 
& = 
\nhsprim{
\frac{\Pht (\bn) - \Pht (b) }{\Pst (\on)}
-
\Pht (b) \frac{\Pst (\on) - \Pst (\om) }{\Pst (\om) \Pst (\on)}
} \\
& \le C \bigg(
\nhsprim{\frac{1}{\Pst (\on)}} \nhsprim{\Pht (\bn) - \Pht (b) } \\
& \quad +
\nhsprim{\Pht (b)} \nhsprim{\frac{1}{\Pst (\on)}} \nhsprim{\frac{1}{\Pst (\om)}} \nhsprim{\Pst (\on) - \Pst (\om) }
\bigg).
}
From (\ref{EstMu1}), (\ref{EstMu3}) we see that $\sup_{n\in \mathbb{N}} \suppess_{t \in [0,T]} \nhsprim{\frac{1}{\Pst (\on(t))}}$ is finite. Also using the same reasoning presented in (\ref{EstMu1}), (\ref{EstMu3}) we can conclude that $\suppess_{t \in [0,T]} \nhsprim{\frac{1}{\Pst (\om(t))}}$ and $\suppess_{t \in [0,T]} \nhsprim{\Pht (b(t))}$ are finite. Thus to prove (\ref{convMun}) it is sufficient to show that
\eqns{
\Pst(\on), \Pht(\bn) & \rightarrow \Pst(\om) , \Pht(b) & \text{strongly in } C(0,T,H^{s^{\prime}}(\td)).
}
This, however, holds based on (\ref{AL2}), (\ref{defPsi}), (\ref{defPhi}) and Lemma \ref{Lem2.5}. \newline
Having convergence results, we may pass to the limit in (\ref{vnEq}) - (\ref{bnEq}). It is easy to see that $v$, $\om$, $b$ satisfy:
\begin{align}
\n{\partial_t v, w} + \n{v \cdot \nab v, w} + \left(  \mu D v, Dw \right) & = 0 
& & \forall w \in H^1_{\divv} (\td), 
\label{vnEqF} \\
\n{\partial_t \om, z} + \n{ v \cdot \nab \om, z}  + \left( \mu \nab \om, \nabla z \right) 
& = - \alpha \n{\om^2, z}
& & \forall z \in H^1 (\td),
\label{omnEqF} \\
\n{\partial_t b, q} + \n{v \cdot \nab b, q} + \left( \mu \nab b, \nabla q \right) 
& = -  \n{ b  \om, q} + \n{\mu |D v |^2, q}
& & \forall q \in H^1 (\td)
\label{bnEqF}
\end{align}
for a.a. $t \in (0,T)$.
We will provide more details for the most troublesome term. First, we wish to establish the convergence 
\eqnsl{
\int_0^T \n{\nn |D\vn|^2, \psi } dt 
\overset{n \rightarrow \infty}{\longrightarrow}
\int_0^T \n{\mu |Dv|^2, \psi } dt,
}{weakFormConv1}
where $\psi \in L^2(0,T,H^1(\td))$. We see that 
\eqns{
\left|\int_0^T \n{\nn |D\vn|^2, \psi } dt 
-
\int_0^T \n{\mu |Dv|^2, \psi } dt
\right|
& \le 
\left|\int_0^T \n{\nn \n{D\vn - Dv} \n{D\vn + Dv}, \psi } dt \right| 
\\ & +
\left|\int_0^T \n{\n{\nn - \mu} |Dv|^2, \psi } dt \right|.
}
Let us first focus on the first term of the right-hand side
\eqns{
\left|\int_0^T \n{\nn D(\vn - v) D(\vn + v), \psi} dt \right| 
& \le 
\int_0^T \nif{\nn} \nif{D(\vn - v)} \nd{D(\vn + v)} \nd{\psi} dt \\
& \le \int_0^T \nhs{\nn} \norm{\vn - v}_{H^{s^\prime + 1}} \norm{\vn + v}_{H^{s^\prime}} \nd{\psi} dt,
} 
where $s^\prime \in \n{\frac{d}{2}, s}$. With the help of H\"older inequality and (\ref{AL1}), (\ref{AL2}), (\ref{convMun}) we get
\eqns{
& \left|\int_0^T \n{\nn D(\vn - v) D(\vn + v), \psi} dt \right| \\
& \quad \quad \quad \le 
\left( \int_0^T \nhsprim{\nn}^2 \nhsprim{\vn + v}^2 \nd{\psi}^2 dt \right)^\jd
\norm{\vn - v}_{L^2(0,T, H^{s^\prime +1})} \overset{n \rightarrow \infty}{\longrightarrow} 0.
} 
By using Lemma \ref{Lem6} we have
\eqns{
\left|\int_0^T \n{\n{\nn - \mu} |Dv|^2, \psi } dt \right| 
& \le 
\int_0^T \nif{\nn - \mu} \nif{Dv} \nd{Dv} \nd{\psi} dt \\
& \le 
\norm{\nn - \mu}_{C(0,T,H^{s^\prime})} \int_0^T \nhsj{v} \nhs{v} \nd{\psi} dt.
}
Thus from (\ref{AL2}), (\ref{ConvFunc3}), (\ref{ConvFunc1}) we get 
\eqns{
\left|\int_0^T \n{\n{\nn - \mu} |Dv|^2, \psi } dt \right| \overset{n \rightarrow \infty}{\longrightarrow} 0
}
and thus (\ref{weakFormConv1}) is proven. \newline
Following the same procedure as presented in \cite{KoKu} or \cite{BuM} one can show that:
\eqns{
\omt \le \omega(t,x) \le \ommt ~~~~\text{for a.e.}~~ (x,t) \in \td \times [0,T]
}
and
\eqns{
\bmt \le b  ~~~~\text{for a.e.}~~ (x,t) \in \td \times [0,T].
} 
Thus due to Definitions (\ref{defPhi}) and (\ref{defPsi}) we have 
\eqns{
\mu = \frac{b}{\om}
}
and thus $(v, \om, b)$ solve system (\ref{divEq})-(\ref{bEq}).
Now we will show the continuity of the solution in $\hs$ norm. It is clear that $\left[ H^{s-1}(\td), H^{s+1}(\td) \right]_\frac{1}{2} = \hs$. Thus from (\ref{ConvFunc2}), (\ref{ConvDT}) and Lions-Magenes Lemma (see Theorem II.5.14 from \cite{Boye2013}) we can conclude 
that 
\eqns{
(v,\om,b) \in C\n{[0,T],\hs}.
}
The uniqueness of the obtained solution easily follows from Theorem \ref{Th1b}.
\section{Proof of Theorem \ref{Th1b}}
First, we would like to conclude that
\eqnsl{
\omt \le \omega_i(t,x) \le \ommt, ~~~
\bmt \le b_i(t,x)  ~~~\text{for a.e.}~~~ (x,t) \in \td \times [0,T], ~~~j=1,2,
}{proofUniq:1}
to have proper bounds on viscosity term $\mu_i = \frac{b_i}{\om_i}$. We see that $b_i$ and $\om_i$ are continuous functions based on Lemma \ref{Lem6}.
Suppose that $\omt > \omega_i(t^*,x^*)$ for some $(t^*,x^*) \in [0,T] \times \td$ such that $\om_i(t,x)> C^*_\omega > 0$ for all $(t,x) \in [0,t^{*}]\times \td$. Then, by following the procedure presented in \cite{KoKu} (starting from Eq.: 96) we obtain the contradiction.
The same reasoning can be conducted for $b_i$. To prove the upper bound on $\om_i$ we proceed as in \cite{KoKu} (starting from Eq.: 98). \\ 
Let us denote by $\delta_v = v_2 - v_1 $, $\delta_\om = \om_2 - \om_1 $, $\delta_b = b_2 - b_1 $. Thus differences $(\delta_v, \delta_\om, \delta_b)$ satisfy the following system of equations  
\eqnsl{
(\partial_t \delta_v, w) + \left( \frac{b_1}{\om_2} D \delta_v, Dw \right) 
& = 
- (v_2 \nab \delta_v, w) - (\delta_v \nabla v_1, w)
- \left( \frac{\delta_b}{\om_2} D v_2, Dw \right) 
\\ & \phantom{=}
+ \left( \frac{b_1 \delta_\om}{\om_1 \om_2} D v_1, Dw \right),
}{deltavEq}
\eqnsl{
(\partial_t \delta_\om, z) + \left( \frac{b_1}{\om_2} \nab \delta_\om, \nab z \right)  
& =
- (v_2 \nab \delta_\om, z) 
- (\delta_v \nabla \om_1, z)
- \left( \frac{\delta_b}{\om_2} \nab \om_2, \nab z \right)
\\ & \phantom{=}
+ \left( \frac{b_1 \delta_\om}{\om_1 \om_2} \nab \om_1, \nab z \right) 
  -\alpha (\delta_\om (\om_2 + \om_1), z),
}{deltaoEq}
\eqnsl{
~~(\partial_t \delta_b, z) + \left( \frac{b_1}{\om_2} \nab \delta_b, \nab z \right)
& =
- (v_2 \nab \delta_b, z) - (\delta_v \nabla b_1, z)
- \left( \frac{\delta_b}{\om_2} \nab b_2, \nab z \right)
\\ & \phantom{=} 
+ \left( \frac{b_1 \delta_\om}{\om_1 \om_2} \nab b_1, \nab z \right) 
 -(\delta_b \om_2, z) - (b_1 \delta_\om , z)
\\ & \phantom{=}
+ \left( \frac{\delta_b}{\om_2} |D v_2|^2, z\right)
+ \left( \frac{b_1}{\om_2} D \delta_v (D v_2 + D v_1), z\right)
\\ & \phantom{=}
- \left( \frac{b_1 \delta_\om}{\om_2 \om_1} |D v_1|^2, z\right).
}{deltabEq}
Now, we test equation (\ref{deltavEq})-(\ref{deltabEq}) with $\delta_v$, $\delta_\om$, $\delta_b$ to obtain estimate of differences in $L^2$ norm. We will show the procedure for $\delta_b$, rest are analogous
\eqns{
\jd \ddt \ndk{ \delta_b} &  + \left( \frac{b_1}{\om_2} \nab \delta_b, \nab \delta_b \right) = - (v_2 \nab \delta_b, \delta_b) - (\delta_v \nabla b_1, \delta_b)
- \left( \frac{\delta_b}{\om_2} \nab b_2, \nab \delta_b \right)
\\ &
+ \left( \frac{b_1 \delta_\om}{\om_1 \om_2} \nab b_1, \nab \delta_b \right)
- (\delta_b \om_2, \delta_b) - (b_1 \delta_\om , \delta_b)
+ \left( \frac{\delta_b}{\om_2} |D v_2|^2, \delta_b\right)
\\ &
+ \left( \frac{b_1}{\om_2} D \delta_v (D v_2 + D v_1), \delta_b\right)
- \left( \frac{b_1 \delta_\om}{\om_2 \om_1} |D v_1|^2, \delta_b\right) .
}
Using (\ref{proofUniq:1}) and H\"older inequality we get
\eqns{
\jd \ddt \ndk{ \delta_b} 
& + \mnit \ndk{\nab \delta_b}
\le 
  \nd{\delta_v} \nif{\nabla b_1} \nd{\delta_b}
+ \nif{\frac{1}{\om_2}} \nd{\delta_b} \nif{\nab b_2} \nd{\nab \delta_b}
\\ & 
+ \nif{b_1} \nd{ \delta_\om} \nif{\frac{1}{\om_1 \om_2}} \nif{\nab b_1} \nd{\nab \delta_b } 
+ \nif{b_1} \nd{\delta_\om} \nd{\delta_b} 
\\ &
+ \nif{\frac{1}{\om_2}} \nif{D v_2}^2 \ndk{\delta_b}
+ \nif{\frac{b_1}{\om_2}} \nd{D \delta_v} \nif{D v_2 + D v_1} \nd{ \delta_b } \\
& + \nif{\frac{1}{\om_2 \om_1}} \nif{b_1} \nif{D v_1}^2  \nd{\delta_\om} \nd{\delta_b }. 
}
Using (\ref{proofUniq:1}) and Lemma \ref{Lem6} we get 
\eqns{
\jd \ddt \ndk{ \delta_b} & + \mnit \ndk{\nab \delta_b}
\le 
\nhsj{b_1} \left(\ndk{\delta_v} 
+ \ndk{\delta_b} \right)
+ C(\omt) \nhsj{b_2}^2 \ndk{\delta_b} 
\\ & 
+ \frac{\mnit}{6} \ndk{\nab \delta_b} 
+ C(\omt) \nhs{b_1}^2  \nhsj{b_1}^2 \ndk{ \delta_\om} 
+ \frac{\mnit}{6} \ndk{\nab \delta_b }
\\ & 
+ \nhs{b_1} \left(\ndk{\delta_\om} + \ndk{\delta_b} \right) 
+ C(\omt) \nhsj{v_2}^2 \ndk{\delta_b} \\
& + C(\omt)\nhs{b_1}^2  \left(\nhsj{v_2}^2 + \nhsj{v_1}^2 \right) \ndk{ \delta_b } + \frac{\mnit}{3} \ndk{D \delta_v} \\
& + C(\omt) \nhs{b_1} \nhsj{v_1}^2  \left( \ndk{\delta_\om} + \ndk{\delta_b } \right) .
}
Thus using the information about the regularity of solutions we can write
\eqnsl{
\jd \ddt \ndk{ \delta_b} 
+ \mnit \ndk{\nab \delta_b}
& \le G_b(t) \left(\ndk{\delta_v} + \ndk{\delta_\om} + \ndk{\delta_b} \right)
\\ & \phantom{\le}
+ \frac{\mnit}{3} \left( \ndk{\nab \delta_v } + \ndk{\nab \delta_\om } + \ndk{\nab \delta_b } \right), 
}{deltabEst}
where function belongs $G_b \in L^1(0,T)$. Analogously there exist functions $G_\om, G_v \in L^1(0,T)$ such that
\eqnsl{
\jd \ddt \ndk{ \delta_\om} 
+ \mnit \ndk{\nab \delta_\om}
& \le G_\om(t) \left(\ndk{\delta_v} + \ndk{\delta_\om} + \ndk{\delta_b} \right)
\\ & \phantom{\le}
+ \frac{\mnit}{3} \left( \ndk{\nab \delta_v } + \ndk{\nab \delta_\om } + \ndk{\nab \delta_b } \right), 
}{deltaoEst}
\eqnsl{
\jd \ddt \ndk{ \delta_v} 
+ \mnit \ndk{\nab \delta_v}
& \le G_v(t) \left(\ndk{\delta_v} + \ndk{\delta_\om} + \ndk{\delta_b} \right)
\\ & \phantom{\le}
+ \frac{\mnit}{3} \left( \ndk{\nab \delta_v } + \ndk{\nab \delta_\om } + \ndk{\nab \delta_b } \right).
}{deltavEst}
By summing (\ref{deltavEst}), (\ref{deltaoEst}) and (\ref{deltabEst}) we get 
\eqns{
\jd \ddt \left(\ndk{\delta_v} + \ndk{\delta_\om} + \ndk{\delta_b} \right) 
\le G(t) \left(\ndk{\delta_v} + \ndk{\delta_\om} + \ndk{\delta_b} \right), 
}
where $G = G_v + G_\om + G_b \in L^1(0,T)$. From Gr\"ownwall inequality follows that $\delta_v(t,x) = 0$, $\delta_\om(t,x) = 0$, $\delta_b(t,x) = 0$ for a.a. $(x,t) \in \td \times [0,T]$. This concludes the proof of uniqueness. 

\subsubsection*{Acknowledgements}
The research was funded by (POB Cybersecurity and data analysis) of Warsaw University of Technology within the Excellence Initiative: Research University (IDUB) programme.
Additionally, the author would like to thank dr Adam Kubica for insightful discussions, which helped with writing this publication.

\section{Appendix I}
To prove Lemma \ref{Lem7} we will utilise some results from pseudo-differential operator theory. In the following definitions, we introduce the needed apparatus. 
\subsection{Definitions and theorems of pseudo-differential operator theory}
\begin{de}[see \cite{Card2019}] \label{defOperatorTm}
Let $m: \mathbb{T}^d \times \mathbb{Z}^{dr} \rightarrow \mathbb{C}$ is a measurable function usually referred to as a symbol. Then, the periodic multi-linear pseudo-differential operator associated with a symbol $m$ is the multilinear operator defined by 
\eqns{
T_m(f)(x) = \sum_{\xi \in \mathbb{Z}^{dr}} e^{i 2 \pi \langle x , \xi_1 + \xi_2 + \dots + \xi_r \rangle} m(x, \xi) 
\hat{f_1}(\xi_1)
\hat{f_2}(\xi_2)
\dots
\hat{f_r}(\xi_r),
}
where $x\in \mathbb{T}^d$, $\xi = (\xi_1, \xi_2, \dots, \xi_r)$, $f = (f_1, f_2, \dots, f_r) \in \mathcal{D}(\mathbb{T}^d)^r$ and
\eqns{
\hat{f}(\xi_i) = 
\int_{\mathbb{T}^d} e^{- i 2 \pi \langle x , \xi_i \rangle} f_i(x) dx
}
\end{de}
\begin{de}[see Definition 3.3.1 in \cite{Ruzh2010}] \label{defDifferenceOperator}
Let $\sigma:\mathbb{Z}^d \rightarrow \mathbb{C}$ and $1 \le i,j\le d$. Let $\delta_j \in \mathbb{N}^d$ be defined by
\eqns{
\n{\delta_j}_i
 = \left\{
\begin{array}{cll}
1, & \m{ if } i = j, \\ 
0, & \m{ if } i \neq j. \\
\end{array}  \right. 
}
We define the forward difference operator $\Delta_{\xi_j} $ by
\eqns{
\Delta_{\xi_j}\sigma (\xi) := \sigma(\xi + \delta_j) - \sigma(\xi)
}
and for $\alpha \in \mathbb{N}^d$ we define
\eqns{
\Delta_{\xi}^\alpha := \Delta_{\xi_1}^{\alpha_1} \dots \Delta_{\xi_d}^{\alpha_d}.
}
\end{de}
\begin{theorem}[see Theorem 3.3 in \cite{Card2019}] \label{thCoifMay}
Assume that $m: \mathbb{T}^d \times \mathbb{Z}^{dr} \rightarrow \mathbb{C}$ is a measurable function that satisfies the discrete symbol inequalities 
\eqnsl{
\sup_{x \in \mathbb{T}^d} 
\left| 
\Delta_{\xi_1}^{\alpha_1}
\Delta_{\xi_2}^{\alpha_2}
\dots
\Delta_{\xi_r}^{\alpha_r}
m(x, \xi_1, \xi_2, \dots, \xi_r)
\right|
\le 
\frac{C_\alpha}{(1+|\xi_1|^2 + \dots +|\xi_r|^2)^{\frac{|\alpha|}{2}}}
}{differenceCondition} 
for all $|\alpha| = |\alpha_1| + |\alpha_2| + \dots + |\alpha_r| \le \left[ \frac{3dr}{2} \right] + 1$. Then, the periodic multi-linear pseudo-differential operator $T_m$ (see Definition \ref{defOperatorTm})
extends to a bounded operator form $L^{p_1}(\mathbb{T}^d)\times L^{p_2}(\mathbb{T}^d)\times \dots L^{p_r}(\mathbb{T}^d)$ into $L^{p}(\mathbb{T}^d)$ provided that 
\eqns{
\frac{1}{p} = \frac{1}{p_1} + \frac{1}{p_2} + \dots \frac{1}{p_r},
~~ 1 < p < \infty,
~~ 1 < p_i \le \infty. 
}
\end{theorem}
\begin{lem} [see Corollary 4.5.7 in \cite{Ruzh2010}] \label{differentialCondToDifference}
Let $0 < \delta \le 1$, $ 0 \le \rho \le 1$. Let $a:\mathbb{T}^d \times \mathbb{R}^d \rightarrow \mathbb{C}$ satisfying
\eqnsl{
\left| 
\partial_\xi^\alpha \partial_x^\beta a(x,\xi) 
\right|
\le \frac{C_{a \alpha \beta m}}{\left( 1 + |\xi|^2 \right)^\frac{m - \rho |\alpha| + \delta |\beta|}{2}}
~~ \forall x \in \mathbb{T}^d, \xi \in \mathbb{R}^d 
}{differentialCondition}
for $|\alpha| \le N_1$ and $|\beta| \le N_2$. Then the restriction $\overline{a} = a |_{\mathbb{T}^d \times \mathbb{Z}^d} $ satisfies the estimate 
\eqns{
\left| 
\Delta_\xi^\alpha \partial_x^\beta \overline{a}(x,\xi) 
\right|
\le \frac{C_{a \alpha \beta m}^\prime C_{a \alpha \beta m}}{\left( 1 + |\xi|^2 \right)^\frac{m - \rho |\alpha| + \delta |\beta|}{2}}
~~ \forall x \in \mathbb{T}^d, \xi \in \mathbb{Z}^d 
}
for $|\alpha| \le N_1$ and $|\beta| \le N_2$.
\end{lem}
\noindent
In the proof of Lemma \ref{Lem7} we will need some results from the theory of interpolation. First, let us introduce the needed definitions.
\begin{de}
We define complex strip in the following way:
\eqns{
S = \{ z \in \mathbb{C}: 0 < \im z < 1 \}.
}
\end{de}
\begin{de}[see \cite{Graf2014}] \label{admissibleGrowth}
A continuous function $F: \overline{S} \rightarrow \mathbb{C}$, which is analytic in $S$ is said to be of admissible growth if there is $0\le \alpha < \pi$ such that 
\eqns{
\sup_{z \in \overline{S}} \frac{\log |F(z)|}{e^{\alpha |\im z|}} < \infty
} 
\end{de}
\begin{de}[see \cite{Graf2014}] \label{analyticFamily}
Let $(\Omega, \Sigma, \mu)$ be a measure space and let $\mathcal{X}_1, ..., \mathcal{X}_m$ be linear spaces. Let us assume that for every $z \in \overline{S}$ there is linear operator $T_z: \mathcal{X}_1 \times \dots \mathcal{X}_m \rightarrow \overline{L}^0(\mu)$, where $\overline{L}^0(\mu)$ denotes the space of all equivalence classes of complex-valued measurable functions on $\Omega$ with the topology of convergence in measure on $\mu$-finite sets. The family $\{ T_z \}_{z \in \overline{S}} $ is said to be analytic if for any $(x_1, ..., x_m) \in \mathcal{X}_1 \times \dots \mathcal{X}_m$ and for almost every $\omega \in \Omega$ the function 
\eqnsl{
\overline{S} \ni z \longmapsto T_z(x_1, ..., x_m)(\omega),
}{analyticDefEq1}
is analytic in $S$ and continuous in $\overline{S}$. Additionally, if for $j=0,1$ the function 
\eqns{
\mathbb{R} \times \Omega \ni (t, \omega) \longmapsto T_{j+it}(x_1, ..., x_m)(\omega)
}
is $(\mathcal{L} \times \Sigma)$-measurable for every  $(x_1, ..., x_m) \in \mathcal{X}_1 \times \dots \mathcal{X}_m$, and for almost every $\omega \in \Omega$ the function (\ref{analyticDefEq1}) is of admissible growth, then the family $\{ T_z \}_{z \in \overline{S}} $ is said to be an admissible analytic family. 
\end{de}
The theorem we are about to cite is more general than stated below. The statement has been adapted to better fit the case at hand.
\begin{theorem}[see Theorem 4.1 in \cite{Graf2014}] \label{GrafacosInterpolationTh}
For $1 \le k \le m$, fix $1 < q_0,q_1,q_{0k},q_{1k} < \infty$ and for $0 < \theta < 1$ define $q$, $q_k$ by setting 
\eqns{
\frac{1}{q_k} = \frac{1 - \theta}{q_{0k}} + \frac{\theta}{q_{1k}}, \hd \hd
\frac{1}{q} = \frac{1 - \theta}{q_{0}} + \frac{\theta}{q_{1}}.
}
Assume that $\mathcal{X}_k$ is a dense linear subspace of $L^{q_{0k}}(\td) \cap L^{q_{1k}}(\td)$ and that $\{ T_z \}_{z \in \overline{S}}$ is an admissible analytic family of multilinear operators $T_z: \mathcal{X}_1 \times \dots \times \mathcal{X}_m  \rightarrow L^{q_{0}}(\td) \cap L^{q_{1}}(\td)$. Suppose that for every $(h_1, ..., h_m) \in \mathcal{X}_1 \times \dots \mathcal{X}_m $, $t\in \mathbb{R}$ and $j=0,1$, we have
\eqnsl{
\norm{T_{j+it}(h_1, ..., h_m)}_{L^{q_j}(\td)} 
\le K_j(t) \norm{h_1}_{L^{q_{j1}}(\td)} \dots \norm{h_m}_{L^{q_{jm}}(\td)},
}{thGrafGrowthConOnLine}
where $K_j$ are Lebesgue measurable functions such that $K_j \in L^1(P_j(\theta, \cdot)dt)$ for all $\theta \in (0,1)$, where
\eqns{
P_j(x+iy, t) = \frac{e^{-\pi (t-y)} \sin \pi x}{ \sin^2 \pi x + (\cos \pi x - (-1)^j e^{-\pi(t-y)})^2}, ~~~~
x + iy \in \overline{S}.
}
Then for all $(f_1, ..., f_m) \in \mathcal{X}_1 \times \dots \mathcal{X}_m $, $0 < \theta < 1$, and $s \in \mathbb{R}$ we have
\eqns{
\norm{T_{\theta+is}(f_1, ..., f_m)}_{L^{q}(\td)} 
\le \n{\frac{q_0}{q_0 - 1}}^{1 - \theta} \n{\frac{q_1}{q_1 - 1}}^\theta
 K_\theta(s) \prod_{j=1}^m \norm{f_j}_{L^{q_{j}}(\td)},
}
where 
\eqns{
\log K_\theta (s) = 
\int_\mathbb{R} P_0(\theta, t) \log K_0(t+s) dt
+
\int_\mathbb{R} P_1(\theta, t) \log K_1(t+s) dt.
}
\end{theorem} 
\begin{rem} \label{pjDecayRate}
For fixed $x \in (0,1)$ and $y \in \mathbb{R}$ there exists constant $C_{x,y}>0$ such that
\eqns{
|P_j(x+iy, t)| \le C_{x,y} e^{-\pi |t|} \quad \forall t \in \mathbb{R}.
}
\end{rem}
\subsection{Proof of Lemma \ref{Lem7}}
The presented proof follows the original proof in the work of Kato and Ponce \cite{Kato1988}. Some of the more calculation-focused lemmas were moved to Appendix II to provide a clearer argument. Additionally, in the presented proof term $4\pi^2$ will be omitted in the definition of $J^s$ to shorten a bit obtained formulas. 
\begin{proof}[Proof of Theorem \ref{Lem7}]
For smooth functions, any considered infinite series in this proof will be convergent as the following holds
\eqnsl{
\forall \phi \in C^\infty(\mathbb{T}^d) \quad 
\forall m \in \mathbb{N} \quad \exists C_{\phi,m} \text{ such that } \forall \xi \in \mathbb{Z}^d
\quad
|\hat{\phi}(\xi)| \le \frac{C_{\phi,m}}{(1+|\xi|^2)^{m/2}}.
}{fourierTransformDecay} 
Let us start the proof by rewriting the expression under the norm using Definition \ref{defJs}: 
\eqns{
J^s(f g)(x) - f J^s(g)(x) 
& = 
 \sum_{\xi \in \mathbb{Z}^d} e^{i 2\pi \langle x , \xi \rangle } \n{1 + |\xi|^2}^{s/2} \widehat{f g}(\xi) \\
& -
f(x)  \sum_{\eta \in \mathbb{Z}^d} e^{i 2\pi \langle x , \eta \rangle } \n{1 + |\eta|^2}^{s/2} \hat{g}(\eta).
}
Now, we use the fact that the Fourier transform of a product is a convolution of transforms
\eqns{
J^s(f g)(x) - f J^s(g)(x) 
& = 
\sum_{\xi \in \mathbb{Z}^d} e^{i 2\pi \langle x , \xi \rangle } \n{1 + |\xi|^2}^{s/2} \sum_{\eta \in \mathbb{Z}^d} \hat{f}(\eta) \hat{g}(\xi - \eta) \\
& -
\sum_{\xi \in \mathbb{Z}^d} e^{i 2\pi \langle x , \xi \rangle } \hat{f}(\xi)
\sum_{\eta \in \mathbb{Z}^d} e^{i 2\pi \langle x , \eta \rangle } \n{1 + |\eta|^2}^{s/2} \hat{g}(\eta).
}
We change the variables in the first integral on the right-hand side $\overline{\xi} = \xi - \eta$: 
\eqns{
J^s(f g)(x) - f J^s(g)(x) 
& = 
\sum_{\eta \in \mathbb{Z}^d}\sum_{\overline{\xi} \in \mathbb{Z}^d} e^{i 2\pi \langle x , \overline{\xi} + \eta \rangle } \n{1 + |\overline{\xi} + \eta|^2}^{s/2} \hat{f}(\eta) \hat{g}(\overline{\xi}) \\
& -
\sum_{\eta \in \mathbb{Z}^d} \sum_{\xi \in \mathbb{Z}^d} e^{i 2\pi \langle x , \xi + \eta \rangle }  \n{1 + |\eta|^2}^{s/2} \hat{f}(\xi) \hat{g}(\eta).
}
We can write this in the following way
\eqnsl{
J^s(f g)(x)& - f J^s(g)(x) \\
&  = 
\sum_{\eta \in \mathbb{Z}^d}\sum_{\xi \in \mathbb{Z}^d} e^{i 2\pi \langle x , \xi + \eta \rangle } \n{ 
\n{1 + |\xi + \eta|^2}^{s/2} - \n{1 + |\eta|^2}^{s/2}
} \hat{f}(\xi) \hat{g}(\eta).
}{X3}
Now, we aim to rewrite obtained expression as a sum of three terms. To do this we introduce the following partition of unity: let  
$\{\Phi_j\}_{j=1}^3 \subset C^\infty(\mathbb{R})$ be such that
\eqns{
0 \le \Phi_j \le 1~~~\text{for } j=1,2,3,
}
\eqns{
\Phi_1 +\Phi_2 +\Phi_3 = 1 \text{   on } [0,\infty) ,
}
\eqnsl{
\supp \Phi_1 \subset \n{- \frac{1}{9}, \frac{1}{9}},~~~~ 
\supp \Phi_2 \subset \n{\frac{1}{10}, 10},~~~~
\supp \Phi_3 \subset \n{9, \infty}.
}{commutatorProofEq0}
The value $-\frac{1}{9}$ in the definition of the $\Phi_1$ actually can be replaced by any negative value. Now, we can write
\eqnsl{
J^s(f g)(x) - f J^s(g)(x) 
= 
\sum_{j=1}^3 \sigma_j (D) (f,g)(x), 
}{X3__}
where
\eqns{
\sigma_j & (D) (f,g)(x) =  \sum_{\eta \in \mathbb{Z}^d}\sum_{\xi \in \mathbb{Z}^d} e^{i 2\pi \langle x , \xi + \eta \rangle } 
\sigma_j (\xi, \eta)
 \hat{f}(\xi) \hat{g}(\eta)
}
and
\eqns{
\sigma_j (\xi, \eta)  =   \n{ 
\n{1 + |\xi + \eta|^2}^{s/2} - \n{1 + |\eta|^2}^{s/2}
}  \Phi_j \left( \frac{1 + |\xi|^2}{1 + |\eta|^2} \right).
} 
Now, we aim to provide the estimate for each term $\sigma_j (D) (f,g)$. For the reader's convenience, each estimate will be obtained in a separate subsection.
\subsubsection{Step 1: Estimate of $\sigma_1 (D) (f,g)$}
We start with transforming function $\sigma_1(\xi,\eta)$ in the following way
\eqns{
\sigma_1(\xi, \eta) 
& = 
\n{1+|\eta|^2}^{s/2} \n{ \n{\frac{1 + |\xi + \eta|^2}{1 + |\eta|^2}}^{s/2} - 1}
\Phi_1\left( \frac{1 + |\xi|^2}{1 + |\eta|^2} \right) \\
& = 
\n{1+|\eta|^2}^{s/2} \n{ \left[1 + (1+|\eta|^2)^{-1} \langle \xi, \xi + 2\eta \rangle \right]^{s/2} - 1}
\Phi_1\left( \frac{1 + |\xi|^2}{1 + |\eta|^2} \right).
}
Our goal is to show that $\sigma_{1}$ after some transformations satisfies condition (\ref{differenceCondition}). However, checking condition (\ref{differenceCondition}) can be troublesome, and instead, we will verify condition (\ref{differentialCondition}) and use Lemma \ref{differentialCondToDifference}. From the organisation point of view, it is easier to first check condition (\ref{differentialCondition}) for $(\xi, \eta)$ such that $ \frac{1 + |\xi|^2}{1 + |\eta|^2} < \frac{1}{9}$ (compare with (\ref{commutatorProofEq0})). 
Now, we perform Taylor expansion of term $\left[ 1 + (1+|\eta|^2)^{-1} \langle \xi, \xi + 2\eta \rangle \right]^{s/2}$. To do this we recall that $(1+x)^\alpha = \sum_{i=0}^\infty \binom{\alpha}{i} x^i$ for $|x| < 1$, where $\binom{\alpha}{i} = \prod_{r=1}^i \frac{\alpha - r + 1}{r}$, $\binom{\alpha}{0} = 1$. Indeed, based on Lemma \ref{ineqalityle1} from Appendix II and the fact that $\frac{1 + |\xi|^2}{1 + |\eta|^2} < \frac{1}{9}$ we have $\left|(1+|\eta|^2)^{-1} \langle \xi, \xi + 2\eta \rangle \right| < 1$. 
Thus we may write 
\eqns{
\sigma_1(\xi, \eta) 
& = 
\n{1+|\eta|^2}^{s/2} \left[\sum_{r=0}^\infty \binom{s/2}{r} (1+|\eta|^2)^{-r} \langle \xi, \xi + 2\eta \rangle^r -1\right]
\Phi_1\left( \frac{1 + |\xi|^2}{1 + |\eta|^2} \right) \\
& = 
\sum_{r=1}^\infty  \binom{s/2}{r} (1+|\eta|^2)^{s/2-r} \langle \xi, \xi + 2\eta \rangle^r \Phi_1\left( \frac{1 + |\xi|^2}{1 + |\eta|^2} \right).
}
Now, we aim to reformulate the terms under the sum. Let us recall that 
\eqnsl{
 (J^{s-1}g)^{\widehat{}} (\eta) = (1+|\eta|^2)^\frac{s-1}{2} \hat{g} (\eta), 
~~~~
 (\partial f)^{\widehat{}} (\xi) = \xi \hat{f}(\xi).  
}{JsProps}
Thus we may write for $(\xi, \eta)$ such that $ \frac{1 + |\xi|^2}{1 + |\eta|^2} < \frac{1}{9}$:
\eqnsl{
\hat{f}(\xi) \hat{g}(\eta) \sigma_1(\xi, \eta) 
& = 
\sum_{r=1}^\infty \langle \sigma_{1,r}, (\partial f)^{\widehat{}}(\xi) \rangle (J^{s-1}g)^{\widehat{}}(\eta) \\
& \equiv 
 \langle \overline{\sigma}_{1}(\xi, \eta), (\partial f)^{\widehat{}}(\xi) \rangle (J^{s-1}g)^{\widehat{}}(\eta),
}{commutatorStep1Eq1.01}
where 
\eqnsl{
\sigma_{1,r} (\xi, \eta) = \binom{s/2}{r} (1+|\eta|^2)^{-r+1/2} \langle \xi, \xi + 2\eta \rangle^{r-1} (\xi + 2 \eta)  \Phi_1 \left(\frac{1 + |\xi|^2}{1+|\eta|^2}\right).
}{commutatorStep1Eq1.1}
For $(\xi, \eta)$ such that $ \frac{1 + |\xi|^2}{1 + |\eta|^2} \geq \frac{1}{9}$ things are simpler:
\eqns{
\sigma_1(\xi, \eta) \hat{f}(\xi) \hat{g}(\eta)  
& = 
0 \cdot \hat{f}(\xi) \hat{g}(\eta)  
 = 
\langle 0, (\partial f)^{\widehat{}}(\xi) \rangle (J^{s-1}g)^{\widehat{}}(\eta) \\
& \equiv 
\langle \overline{\sigma}_1(\xi,\eta), (\partial f)^{\widehat{}}(\xi) \rangle (J^{s-1}g)^{\widehat{}}(\eta).
}
With this, we can conclude that 
\eqnsl{
\overline{\sigma}_1(\xi, \eta) = \left\{
\begin{array}{cll}
\sum_{r=1}^\infty \sigma_{1,r}(\xi,\eta) & \m{ for } (\xi, \eta): \frac{1 + |\xi|^2}{1 + |\eta|^2} < \frac{1}{9} \\ 
0 &  \m{ otherwise } \\
\end{array}.  \right. 
}{commutatorStep1Eq1.2}
As mentioned before, we will show that for each of $r$ function $\sigma_{1,r}$ fulfils condition (\ref{differentialCondition}) up to some number of differentiations $k(d) \in \mathbb{N}$. We will analyse $\sigma_{1,r}$ piece by piece. Let $m \geq 0$. 
Then, based on Lemma \ref{technicalLemma1} from Appendix II let us observe that for $\alpha_i \in \mathbb{N}$ such that $\sum_{i=1}^d \alpha_i = m$ we have 
\eqns{
\left| \frac{\partial^{m} \left[ (1+|\eta|^2)^{-r + 1/2} \right]}{\partial^{\alpha_1}_{\eta_1}\dots\partial^{\alpha_d}_{\eta_d} } \right|
& \le 
C (1+r)^{m} (1+|\eta|^2)^{-r+1/2 - \frac{m}{2}} 
~~~~ \forall (\xi,\eta) \in \mathbb{R}^{2d}.
}
Using the assumption $\frac{1 + |\xi|^2}{1+|\eta|^2} < \frac{1}{9}$ we can write
\eqnsl{
\left| \frac{\partial^{m} \left[ (1+|\eta|^2)^{-r + 1/2} \right]}{\partial^{\alpha_1}_{\eta_1}\dots\partial^{\alpha_d}_{\eta_d} } \right|
& \le 
\frac{C(1+r)^{m}}{(1+ |\xi|^2 + |\eta|^2)^{\frac{m}{2}}(1 + |\eta|^2)^{r-1/2 }} 
~~~ \forall (\xi,\eta): \frac{1 + |\xi|^2}{1+|\eta|^2} < \frac{1}{9}.
}{commutatorStep1Eq2}
Now let us focus on the term $\langle \xi , \xi + 2\eta \rangle^{r-1}(\xi_k + 2\eta_k)$.
Let $\alpha_i, \beta_i \in \mathbb{N}$ such that $\sum_{i=1}^d (\alpha_1 + \beta_i) = m$. Based on Lemma \ref{technicalLemma2} from Appendix II we have 
\eqnsl{
& \left|
\frac{\partial^{m} \left[ \langle \xi , \xi + 2\eta \rangle^{r-1} (\xi_k + 2\eta_k) \right]}{\partial^{\alpha_1}_{\xi_1}\dots\partial^{\alpha_d}_{\xi_d} \partial^{\beta_1}_{\eta_1}\dots\partial^{\beta_d}_{\eta_d} } 
\right| \\
& \quad \quad \quad \quad \le 
C \left( 1 + r \right)^{m} \left(\frac{7}{9} \right)^{r}   \frac{\left(1 + |\eta|^2 \right)^{r - \frac{1}{2}}}{\left(1 + |\xi|^2 + |\eta|^2  \right)^{\frac{m}{2}}}
~~ \forall (\xi,\eta): \frac{1 + |\xi|^2}{1+|\eta|^2} < \frac{1}{9}. 
}{commutatorStep1Eq3}
Now we will handle the last term in the definition of $\sigma_{1,r}$. From Lemma \ref{technicalLemma3} from Appendix II we have that 
\eqnsl{
\left|
\frac{\partial^{m} \left[ \Phi_1 \n{\frac{1 + |\xi|^2}{1+|\eta|^2}}\right]}{\partial^{\alpha_1}_{\xi_1}\dots\partial^{\alpha_d}_{\xi_d} \partial^{\beta_1}_{\eta_1}\dots\partial^{\beta_d}_{\eta_d} } 
\right|
\le
\frac{C}{\left(1 + |\xi|^2 + |\eta|^2  \right)^{\frac{m}{2}}}
~~~~ \forall (\xi,\eta)\in \mathbb{R}^{2d}.
}{commutatorStep1Eq4}
Thus based on (\ref{commutatorStep1Eq1.1}), (\ref{commutatorStep1Eq1.2}), (\ref{commutatorStep1Eq2}), (\ref{commutatorStep1Eq3}), (\ref{commutatorStep1Eq4}) and Lemma \ref{technicalLemma0} from Appendix II we can finally write 
\eqns{
\left| \frac{\partial^{m} \overline{\sigma}_{1}(\eta,\xi)}{\partial^{\alpha_1}_{\xi_1}\dots\partial^{\alpha_d}_{\xi_d} \partial^{\beta_1}_{\eta_1}\dots\partial^{\beta_d}_{\eta_d} } \right| 
 & \le \frac{C(m)}{\left(1 + |\xi|^2 + |\eta|^2  \right)^{\frac{m}{2}}} \sum_{r=1}^\infty \left|\binom{s/2}{r} \right| (1+r)^m 
\left(\frac{7}{9} \right)^{r}.
}
The sum on the right-hand side is finite based on the D'Alembert criterion for series convergence. Indeed we see that
\eqns{
\frac{\left|\binom{s/2}{r+1} \right| (2+r)^m 
\left(\frac{7}{9} \right)^{r+1}}{\left|\binom{s/2}{r} \right| (1+r)^m 
\left(\frac{7}{9} \right)^{r}}
=
\frac{7}{9} \left| \frac{s/2 - (r + 1) + 1}{r} \right| 
\n{\frac{r+2}{r+1}}^m \overset{r \rightarrow \infty}{\longrightarrow} \frac{7}{9} 
< 1.
}
Based on Lemma \ref{differentialCondToDifference} we have 
\eqns{
\left| 
\Delta_{\xi_1}^{\alpha_1} \dots \Delta_{\xi_d}^{\alpha_d} 
\Delta_{\eta_1}^{\beta_1} \dots \Delta_{\eta_d}^{\beta_d} 
\overline{\sigma}_1(\xi, \eta) 
\right|
\le \frac{C(m,s)}{\left( 1 + |\xi|^2 + |\eta|^2 \right)^\frac{m}{2}}.
}
Thus based on (\ref{commutatorStep1Eq1.01}), (\ref{commutatorStep1Eq1.2}) and Theorem \ref{thCoifMay} we have
\eqnsl{
\norm{\sigma_1(D)(f,g)}_p = \norm{\overline{\sigma}_1(D)(\partial f,J^{s-1}g)}_p \le C \norm{\partial f}_{p_1} \norm{J^{s-1} g}_{p_2}.
}{commutatorStep1EqFinal}
\subsubsection{Step 2: Estimate of $\sigma_3 (D) (f,g)$}
Now we will consider the therm $\sigma_3(D)(f,g)$. Firstly, we define 
\eqnsl{
\sigma_{3,1}(\xi,\eta) = \left( \left(1 + |\xi+\eta|^2 \right)^{s/2} - 1\right) \Phi_3 \left(\frac{1 + |\xi|^2}{1+|\eta|^2}\right)
}{commutatorStep2Eqle2}
and 
\eqnsl{
\sigma_{3,2}(\xi,\eta) = \left( \left(1 + |\eta|^2 \right)^{s/2} - 1\right) \Phi_3 \left(\frac{1 + |\xi|^2}{1+|\eta|^2}\right).
}{commutatorStep2Eqle1}
We clearly see that $\sigma_{3} = \sigma_{3,1} - \sigma_{3,2}$. Based on (\ref{JsProps}) we have 
\eqnsl{
\sigma_{3,1}(\xi,\eta) \hat{f}(\xi) \hat{g}(\eta) =  \frac{ \left(1 + |\xi+\eta|^2 \right)^{s/2} - 1}{\left(1 + |\xi|^2 \right)^{s/2}}
\Phi_3 \left(\frac{1 + |\xi|^2}{1+|\eta|^2}\right)
 (J^{s}f)^{\widehat{}} (\xi) \hat{g}(\eta).
}{commutatorStep2Eqle0.5}
Now we have to show that 
\eqnsl{
\sigma_{3,1}^{*}(\xi,\eta) =  \left(1 + |\xi|^2 \right)^{-s/2} \left( \left(1 + |\xi+\eta|^2 \right)^{s/2} - 1\right)  \Phi_3 \left(\frac{1 + |\xi|^2}{1+|\eta|^2}\right)
}{commutatorStep2Eqle0.51}
fulfils condition (\ref{differenceCondition}). Like previously, we will show that condition (\ref{differentialCondition}) holds and deduce (\ref{differenceCondition}) from Lemma \ref{differentialCondToDifference}. Also as before we will split our considerations into two cases: for $(\xi, \eta)$ such that $\frac{1 + |\xi|^2}{1+|\eta|^2} > 9$ and the opposite. We start with the prior case. Based on Lemma \ref{technicalLemma1} from Appendix II for $\alpha_i \in \mathbb{N}$ such that $\sum_{i=1}^d \alpha_i = m$ we can deduce the following 
\eqns{
\left| \frac{\partial^{m} \left[ (1+|\xi|^2)^{-s/2} \right]}{\partial^{\alpha_1}_{\xi_1}\dots\partial^{\alpha_d}_{\xi_d} } \right|
 \le 
C(s,m) (1+|\xi|^2)^{-s/2} \n{1 + |\xi|^2 }^{-\frac{m}{2}}
~~~~ \forall (\xi,\eta) \in \mathbb{R}^{2d}.
}
Using the assumption $\frac{1+|\xi|^2}{1+|\eta|^2} > 9$ we have
\eqnsl{
\left| \frac{\partial^{m} \left[ (1+|\xi|^2)^{-s/2} \right]}{\partial^{\alpha_1}_{\xi_1}\dots\partial^{\alpha_d}_{\xi_d} } \right|
 \le 
C (1+|\xi|^2)^{-s/2} \n{1 + |\xi|^2 + |\eta|^2 }^{-\frac{m}{2}}
~~~~ \forall (\xi,\eta) : \frac{1+|\xi|^2}{1+|\eta|^2} > 9.
}{commutatorStep2Eq0}
Based on Lemma \ref{technicalLemma4} from Appendix II for  $\alpha_i,\beta_i \in \mathbb{N}$ such that $\sum_{i=1}^d (\alpha_i + \beta_i) = m$ we have 
\eqnsl{
\left| \frac{\partial^{m} \left[ (1+|\xi + \eta|^2)^{s/2} - 1 \right]}{\partial^{\alpha_1}_{\xi_1}\dots\partial^{\alpha_d}_{\xi_d} \partial^{\beta_1}_{\eta_1}\dots\partial^{\beta_d}_{\eta_d} }  \right|
& \le 
C  (1 + |\xi|^2 + |\eta|^2 )^{\frac{s}{2} - \frac{m}{2}}
~~~~ \forall (\xi,\eta) : \frac{1+|\xi|^2}{1+|\eta|^2} > 9.
}{commutatorStep2Eq1}
Based on Lemma \ref{technicalLemma3} from Appendix II we get 
\eqnsl{
\left|
\frac{\partial^{m} \left[ \Phi_3 \n{\frac{1 + |\xi|^2}{1 + |\eta|^2}}\right]}{\partial^{\alpha_1}_{\xi_1}\dots\partial^{\alpha_d}_{\xi_d} \partial^{\beta_1}_{\eta_1}\dots\partial^{\beta_d}_{\eta_d} } 
\right|
\le 
 \frac{C}{\left(1 + |\xi|^2 + |\eta|^2\right)^{\frac{m}{2}}}.
}{commutatorStep2Eq2}
Collecting (\ref{commutatorStep2Eqle0.51}), (\ref{commutatorStep2Eq0}), (\ref{commutatorStep2Eq1}) and (\ref{commutatorStep2Eq2}) and by using Lemma \ref{technicalLemma0} from Appendix II we get 
\eqns{
& \left| \frac{\partial^{m} \sigma_{3,1}^*}{\partial^{\alpha_1}_{\xi_1}\dots\partial^{\alpha_d}_{\xi_d} \partial^{\beta_1}_{\eta_1}\dots\partial^{\beta_d}_{\eta_d}} \right| \\
& \qquad \quad \quad\le 
C \left( \frac{1+|\xi|^2+|\eta|^2}{1+|\xi|^2} \right)^{s/2} \n{1 + |\eta|^2 + |\xi|^2 }^{-\frac{m}{2}}
~\forall (\xi,\eta) : \frac{1+|\xi|^2}{1+|\eta|^2} > 9.
}
We see that thanks to $ \frac{1+|\xi|^2}{1+|\eta|^2} > 9$ we have
\eqns{
\frac{1+|\xi|^2 + |\eta|^2}{1+|\xi|^2}
\le 
1 + \frac{|\eta|^2}{1+|\xi|^2}
\le 
1 + \frac{1}{9}
= \frac{10}{9}.
}
and thus
\eqns{
\left| \frac{\partial^{m} \sigma_{3,1}^* }{\partial^{\alpha_1}_{\xi_1}\dots\partial^{\alpha_d}_{\xi_d} \partial^{\beta_1}_{\eta_1}\dots\partial^{\beta_d}_{\eta_d}} \right|
 \le 
C  \n{1 + |\eta|^2 + |\xi|^2 }^{-\frac{m}{2}}
~~~~ \forall (\xi,\eta) : \frac{1+|\xi|^2}{1+|\eta|^2} > 9.
}
Obtained formula is also valid in case $\frac{1+|\xi|^2}{1+|\eta|^2} \le 9$ due to $\supp \Phi_3 \subset (9,\infty)$. Thus based on Lemma \ref{differentialCondToDifference} we can deduce
\eqns{
\left| 
\Delta_{\xi_1}^{\alpha_1} \dots \Delta_{\xi_d}^{\alpha_d} 
\Delta_{\eta_1}^{\beta_1} \dots \Delta_{\eta_d}^{\beta_d} 
\sigma_{3,1}^*(\xi, \eta)
\right|
\le 
\frac{C_\alpha}{(1+|\xi|^2+|\eta|^2)^{\frac{|\alpha| + |\beta|}{2}}}
} 
for all $\alpha = (\alpha_1, \dots,  \alpha_d) \in \mathbb{N}^d$, $\beta = (\beta_1, \dots,  \beta_d) \in \mathbb{N}^d$. Finally from (\ref{commutatorStep2Eqle0.5}) and Theorem \ref{thCoifMay} we get 
\eqnsl{
\norm{\sigma_{3,1}(D)(f,g)}_p = \norm{\sigma_{3,1}^*(D)(J^s f,g)}_p \le C \norm{g}_{p_3} \norm{J^{s} f}_{p_4}  .
}{commutatorStep2EqPart1Final}
Now let us proceed with $\sigma_{3,2}$. Before we start let us observe that $\sigma_{3,2}(\xi, 0) \equiv 0$. Let us define auxiliary smooth function $\Psi$ such that $ 0 \le \Psi(x) \le 1$, $\Psi(x) = 1$ for $x < 3/4$, $\Psi(x) = 0$ for $x > 9/10$. Then we can rewrite (\ref{commutatorStep2Eqle1}) in a following form 
\eqnsl{
& \sigma_{3,2} (\xi, \eta) \hat{f}(\xi) \hat{g}(\eta) \\
& \quad = \left\{
\begin{array}{cll}
|\xi|^{-2}\left\langle \xi, \n{\partial f}^{\widehat{}}(\xi) \right\rangle
\left\langle \eta, \n{G J^{s-1} g}^{\widehat{}}(\eta) \right\rangle
\Phi_3 \n{\frac{1 + |\xi|^2}{1 + |\eta|^2}} & \m{ for } (\xi, \eta): \frac{1 + |\xi|^2}{1 + |\eta|^2} > 9 \\ 
\left\langle 0, \n{\partial f}^{\widehat{}}(\xi) \right\rangle
\left\langle 0, \n{G J^{s-1} g}^{\widehat{}}(\eta) \right\rangle &  \m{ otherwise } \\
\end{array},  \right. 
}{commutatorStep2Eq-2}
where
\eqnsl{
\n{G h}^{\widehat{}}(\eta)
=
g(\eta) \hat{h} (\eta),
}{Gdef_a}
\eqnsl{
g(\eta)
=
\left\{
\begin{array}{cll}
\eta |\eta|^{-2} 
\frac{\left(1+|\eta|^2\right)^{s/2} - 1 }{\left(1+|\eta|^2\right)^{s/2 - 1/2}} \Psi \n{\frac{1}{1+|\eta|^2}}
& \m{ for } \eta: \frac{1}{1 + |\eta|^2} < \frac{9}{10} \\ 
0  &  \m{ otherwise } \\
\end{array}
\right.
.
}{Gdef_b}
The purpose of term $\Psi \n{\frac{1}{1+|\eta|^2}}$ is to cut-off region near $\eta = 0$, without affecting values for $\eta \in \mathbb{Z}^d \setminus \{0\}$. 
We see that based on Lemma \ref{technicalLemma3} from Appendix II we have 
\eqnsl{
\left|
 \frac{\partial^{m} \left[  \Phi_3 \n{ \frac{1 + |\xi|^2}{1 + |\eta|^2}} \right]}{\partial^{\alpha_1}_{\xi_1}\dots\partial^{\alpha_d}_{\xi_d} \partial^{\beta_1}_{\eta_1}\dots\partial^{\beta_d}_{\eta_d} }
\right| 
\le 
\frac{C}{\n{1 + |\eta|^2 + |\xi|^2}^{\frac{m}{2}}}.
}{commutatorStep2Eq3}
With the use of that fact that $\frac{1 + |\xi|^2}{1 + |\eta|^2} > 9$ (which implies that $|\xi|>2 \sqrt{2}$), we also see that
\eqnsl{
\left|
 \frac{\partial^{m} \left[ |\xi|^{-2} \xi_j \right]}{\partial^{\alpha_1}_{\xi_1}\dots\partial^{\alpha_d}_{\xi_d} \partial^{\beta_1}_{\eta_1}\dots\partial^{\beta_d}_{\eta_d} }
\right|
\le C |\xi|^{-1-m} \le \frac{C}{\n{1 + |\eta|^2 + |\xi|^2}^{\frac{m+1}{2}}}.
}{commutatorStep2Eq4}
We see that
\eqns{
&
 \frac{\partial^{m} \left[ |\xi|^{-2} \eta_k \xi_j \Phi_3 \right]}{\partial^{\alpha_1}_{\xi_1}\dots\partial^{\alpha_d}_{\xi_d} \partial^{\beta_1}_{\eta_1}\dots\partial^{\beta_d}_{\eta_d} }
\\
& ~~~~~~ = 
\eta_k
\frac{\partial^{m} \left[ |\xi|^{-2} \xi_j \Psi_3 \right]}{\partial^{\alpha_1}_{\xi_1}\dots\partial^{\alpha_d}_{\xi_d} \partial^{\beta_{1}}_{\eta_1}\dots\partial^{\beta_{d}}_{\eta_d}}
+
\binom{\beta_k}{1}
\frac{\partial^{m - 1} \left[ |\xi|^{-2} \xi_j \Phi_3 \right]}{\partial^{\alpha_1}_{\xi_1}\dots\partial^{\alpha_d}_{\xi_d} \partial^{\beta_{1}}_{\eta_1} \dots \partial^{\beta_{k-1}}_{\eta_{k-1}} \partial^{\beta_{k} - 1}_{\eta_{k}} \partial^{\beta_{k+1}}_{\eta_{k+1}} \dots \partial^{\beta_{d}}_{\eta_d}}.
}
Thus we see that based on (\ref{commutatorStep2Eq3}), (\ref{commutatorStep2Eq4}) and Lemma \ref{technicalLemma0} from Appendix II we can calculate
\eqns{
\left|
 \frac{\partial^{m} \left[ |\xi|^{-2} \eta_k \xi_j \Phi_3 \right]}{\partial^{\alpha_1}_{\xi_1}\dots\partial^{\alpha_d}_{\xi_d} \partial^{\beta_1}_{\eta_1}\dots\partial^{\beta_d}_{\eta_d} }
\right|
\le 
\frac{C}{\n{1 + |\eta|^2 + |\xi|^2}^{\frac{m}{2}}}.
}
We see that based on the above, (\ref{commutatorStep2Eq-2}), Lemma \ref{differentialCondToDifference} and Theorem \ref{thCoifMay} we may conclude
\eqnsl{
\norm{\sigma_{3,2}(D)(f,g)}_p \le C \norm{\partial f}_{p_1} \norm{G J^{s-1} g }_{p_2}.
}{commutatorStep2Eq4.5}
Now we need to derive the estimate for $\norm{G J^{s-1} g }_{p_2}$. 
We see that based on Lemma \ref{technicalLemma1} from Appendix II we have 
\eqnsl{
\left|
\frac{\partial^{m} \left[ (1+|\eta|^2)^{1/2 - s/2} \right]}{\partial^{\alpha_1}_{\eta_1}\dots\partial^{\alpha_d}_{\eta_d} } 
\right|
\le 
C (1+|\eta|^2)^{1/2-s/2 - m/2}  ~~~~~\forall \eta \in \mathbb{R}^d.
}{commutatorStep2Eq5}
Similarly also based on Lemma \ref{technicalLemma1} from Appendix II we have
\eqnsl{
\left|
\frac{\partial^{m} \left[ (1+|\eta|^2)^{s/2} \right]}{\partial^{\alpha_1}_{\eta_1}\dots\partial^{\alpha_d}_{\eta_d} } 
\right|
\le 
C (1+|\eta|^2)^{s/2-m/2}   ~~~~~\forall \eta \in \mathbb{R}^d .
}{commutatorStep2Eq6}
Also, we have
\eqnsl{
\left|
\frac{\partial^{m} \left[ \eta |\eta|^{-2} \right]}{\partial^{\alpha_1}_{\eta_1}\dots\partial^{\alpha_d}_{\eta_d} } 
\right|
\le 
C |\eta|^{-m-1}  
\le
C\n{1 + |\eta|^2}^{-\frac{m+1}{2}}  ~~~~~\forall \eta \in \mathbb{R}^d : \frac{1}{1 + |\eta|^2} < \frac{9}{10}.
}{commutatorStep2Eq7}
Applying the same reasoning employed in Lemma \ref{technicalLemma3} from Appendix II we get 
\eqnsl{
\left|
\frac{\partial^{m} \left[ \Psi \n{ \frac{1}{1 + |\eta|^2}} \right]}{\partial^{\alpha_1}_{\eta_1}\dots\partial^{\alpha_d}_{\eta_d} } 
\right|
\le 
C\n{1 + |\eta|^2}^{-\frac{m}{2}}    ~~~~~\forall \eta \in \mathbb{R}^d .
}{commutatorStep2Eq8}
We see that based on (\ref{commutatorStep2Eq5}), (\ref{commutatorStep2Eq6}), (\ref{commutatorStep2Eq7}), (\ref{commutatorStep2Eq8}) and Lemma \ref{technicalLemma0} from Appendix II we get
\eqns{
& \left|
\frac{
\partial^{m} \left[ \eta |\eta|^{-2} \left(1+|\eta|^2\right)^{1/2-s/2}
\left( \left(1+|\eta|^2\right)^{s/2} - 1 \right) \Psi \n{ \frac{1}{1 + |\eta|^2}} \right]}{\partial^{\alpha_1}_{\eta_1}\dots\partial^{\alpha_d}_{\eta_d} } 
\right| 
\le C \n{ 1 + |\eta|^2}^{-\frac{m}{2}}.
}
Thus based on the above, (\ref{Gdef_a}), (\ref{Gdef_b}), Lemma \ref{differentialCondToDifference} and Theorem \ref{thCoifMay}  we have 
\eqns{
\norm{G h}_{p_2}
= 
\norm{\n{\n{G h}^{\widehat{}}}^{\widecheck{}}}_{p_2}
=
\norm{\sum_{\eta \in \mathbb{Z}^{d}} e^{i 2 \pi \langle  \cdot , \eta \rangle} g(\eta) 
\hat{h}(\eta)}_{p_2}
\le 
C \norm{h}_{p_2}.
}
Thus we can conclude based on (\ref{commutatorStep2Eq4.5}) we have
\eqnsl{
\norm{\sigma_{3,2}(D)(f,g)}_p 
\le C \norm{\partial f}_{p_1} \norm{J^{s-1}g}_{p_2}.
}{commutatorStep2EqPart2Final}
Thus using (\ref{commutatorStep2EqPart1Final}) and (\ref{commutatorStep2EqPart2Final}) we obtain
\eqnsl{
\norm{\sigma_{3}(D)(f,g)}_p 
\le C \n{ 
\norm{\partial f}_{p_1} \norm{J^{s-1}g}_{p_2} 
+ 
\norm{g}_{p_3} \norm{J^{s}f}_{p_4} 
}.
}{commutatorStep2EqFinal}
\subsubsection{Step 3: Estimate of $\sigma_2(D)(f,g)$}
Now we have to estimate 
\eqns{
\sigma_2(\xi,\eta) = 
\n{ 
\n{1 + |\xi + \eta|^2}^{s/2} - \n{1 + |\eta|^2}^{s/2}
} \Phi_2 \n{ \frac{1 + |\xi|^2}{1 + |\eta|^2} }.
}
To do this we introduce two new functions
\begin{align}
\sigma_{2,1}(\xi,\eta) & =  
\n{1 + |\xi + \eta|^2}^{s/2} \Phi_2 \n{ \frac{1 + |\xi|^2}{1 + |\eta|^2} },
\label{commutatorStep3Eq1} \\
\sigma_{2,2}(\xi,\eta) & =  
\n{1 + |\eta|^2}^{s/2} \Phi_2 \n{ \frac{1 + |\xi|^2}{1 + |\eta|^2} }.
\label{commutatorStep3Eq2} 
\end{align}
It is clear that $\sigma_2 = \sigma_{2,1} - \sigma_{2,2}$. 
We see that the following holds:
\eqns{
\sigma_{2,2}(\xi,\eta) \hat{f}(\xi) \hat{g}(\eta)
& =
\n{1 + |\eta|^2}^{s/2}\n{1 + |\xi|^2}^{-s/2}
\Phi_2\n{ \frac{1 + |\xi|^2}{1 + |\eta|^2} }
(J^{s}f)^{\widehat{}}(\xi)
\hat{g}(\eta) \\
& =
\overline{\sigma}_{2,2}(\xi,\eta)
(J^{s}f)^{\widehat{}}(\xi)
\hat{g}(\eta).
}
Recalling that $\supp \Phi_2 \subset \left( \frac{1}{10} , 10 \right)$ and using Lemmas \ref{technicalLemma1}, \ref{technicalLemma3}, \ref{technicalLemma0} from Appendix II combined with Lemma \ref{differentialCondToDifference} and Theorem \ref{thCoifMay}  it is easy to see that
\eqnsl{
\norm{\sigma_{2,2}(D)(f,g)}_p 
=
\norm{\overline{\sigma}_{2,2}(D)(J^s f,g)}_p 
\le 
C \norm{g}_{p_3} \norm{J^{s}f}_{p_4} .
}{commutatorStep3Eq2.1}
Now we have to provide the estimate for $\sigma_{2,1}$:
\eqnsl{
\sigma_{2,1}(\xi,\eta) \hat{f}(\xi) \hat{g}(\eta)
& =
\n{1 + |\xi + \eta|^2}^{s/2}\n{1 + |\xi|^2}^{-s/2}
\Phi_2\n{ \frac{1 + |\xi|^2}{1 + |\eta|^2} }
(J^{s}f)^{\widehat{}}(\xi)
\hat{g}(\eta) \\
& =
\overline{\sigma}_{2,1}(\xi,\eta)
(J^{s}f)^{\widehat{}}(\xi)
\hat{g}(\eta).
}{commutatorStep3Eq3}
As we see in the formulation of Theorem \ref{thCoifMay} condition (\ref{differenceCondition}) has to be valid up to some number of differences taken. Let us denote this number by $k(d)$. Now, let us analyse the case where $s/2 \geq k(d)$. We try to proceed in the case of $\overline{\sigma}_{2,1}$ in the same way as in the case of $\overline{\sigma}_{2,2}$. Thus we try to validate the assumption (\ref{differentialCondition}) in Lemma \ref{differentialCondToDifference}. While doing so we may have to estimate negative powers of term $1 + |\xi + \eta|^2$, which is problematic. This is not the issue when $s/2 \geq k(d)$ and calculations can be performed similarly to $\overline{\sigma}_{2,2}$ (thanks to Lemma \ref{technicalLemma4} from Appendix II). To obtain the estimate in the case where  $s$ is not so large we will apply the complex interpolation method. Thus we extend the definition of $\overline{\sigma}_{2,1}$, $\sigma_{2,1}$ to complex values:
\begin{align}
\overline{\sigma}_{2,1}^z(\xi,\eta) & = \n{1 + |\xi + \eta|^2}^{z/2}\n{1 + |\xi|^2}^{-z/2}
\Phi_2\n{ \frac{1 + |\xi|^2}{1 + |\eta|^2} }, \label{sigmaExt1}\\
\sigma_{2,1}^z(\xi,\eta) & = \n{1 + |\xi + \eta|^2}^{z/2}
\Phi_2\n{ \frac{1 + |\xi|^2}{1 + |\eta|^2} }, \label{sigmaExt2}
\end{align}
such that $0 \le \mathit{Re} z \le 2 k(d) $. If we chose $z = 2k + it$ we can conclude using (\ref{sigmaExt1}) and Lemmas \ref{technicalLemma1}, \ref{technicalLemma3}, \ref{technicalLemma4}, \ref{technicalLemma0}, \ref{differentialCondToDifference} and Theorem \ref{thCoifMay} that for $\psi, \phi \in C^\infty(\td)$ we have 
\eqnsl{
\norm{\overline{\sigma}_{2,1}^{2k+it}(D)(\phi,\psi)}_p 
\le C(t)  \norm{\psi}_{p_3}\norm{\phi}_{p_4},
}{commutatorStep3Eq3.5}
where $C(t) = \overline{C} \cdot (1+|t|)^k$ (this factor is the result of $k(d)$ differentiations present in Theorem \ref{thCoifMay}). Now we need to establish a similar estimate in the case of $z = it$. To this end, we observe that (based on transformations that lead to (\ref{X3__})) we have
\eqns{
J^{it}(\phi \psi) = \sum_{i=1}^3 \kappa_i^{it}(D)(\phi, \psi),
}
where 
\eqns{
\kappa_j^{it}(\xi,\eta) = 
\n{1 + |\xi + \eta|^2}^{it/2}
 \Phi_j\n{ \frac{1 + |\xi|^2}{1 + |\eta|^2} }.
}
We want to obtain an estimate of $\overline{\sigma}_{2,1}^{it}(D)(\phi, \psi)$, however, it is easier to start with obtaining an estimate for $\sigma_{2,1}^{it}(D)(\phi, \psi)$:
\eqnsl{
\sigma_{2,1}^{it}(D)(\phi, \psi) = \kappa_2^{it}(D)(\phi, \psi) = J^{it}(\phi\psi) - \kappa_1^{it}(D)(\phi, \psi) - \kappa_3^{it}(D)(\phi, \psi).
}{commutatorStep3Eq3.51}
Now we need to derive estimates for each term on the right-hand side. First we will concentrate on $\kappa_j^{it}(\xi,\eta)$ for $j=1,3$. Based on Lemma \ref{technicalLemma4} from Appendix II we have 
\eqns{
& \left|
\frac{\partial^{m} \left[ (1+|\xi + \eta|^2)^{it/2} \right]}{\partial^{\alpha_1}_{\eta_1}\dots\partial^{\alpha_d}_{\eta_d} \partial^{\beta_1}_{\eta_1}\dots\partial^{\beta_d}_{\eta_d} } 
\right| \\
& \qquad \quad \le  
C \cdot (1 + |t|)^m
\n{1 + |\eta|^2 + |\xi|^2 }^{-\frac{m}{2}} \text{ for } \frac{1+|\xi|^2}{1 + |\eta|^2} > 9 \text{ or } \frac{1+|\xi|^2}{1 + |\eta|^2} < \frac{1}{9}.
}
Based on Lemma \ref{technicalLemma3} from Appendix II we have 
\eqns{
\left|
\frac{\partial^{m} \left[ \Phi_j\left(\frac{1+|\xi|^2}{1+ |\eta|^2} \right) \right]}{\partial^{\alpha_1}_{\eta_1}\dots\partial^{\alpha_d}_{\eta_d} \partial^{\beta_1}_{\eta_1}\dots\partial^{\beta_d}_{\eta_d} } 
\right|
\le  
C \cdot (1 + |t|)^m
\n{1 + |\eta|^2 + |\xi|^2 }^{-\frac{m}{2}}.
}
Thus based on Lemma \ref{technicalLemma0} from Appendix II we get 
\eqns{
\left|
\frac{\partial^{m} \left[ \kappa^{it}_j(\xi,\eta) \right]}{\partial^{\alpha_1}_{\eta_1}\dots\partial^{\alpha_d}_{\eta_d} \partial^{\beta_1}_{\eta_1}\dots\partial^{\beta_d}_{\eta_d} } 
\right|
\le  
C \cdot (1 + |t|)^m
\n{1 + |\eta|^2 + |\xi|^2 }^{-\frac{m}{2}}
\text{ for } j = 1,3.
}
Based on Lemma \ref{differentialCondToDifference} and Theorem \ref{thCoifMay} we have
\eqnsl{
\norm{\kappa_{j}^{it}(D)(\phi, \psi)}_p 
\le C(1+|t|)^k \norm{\psi}_{p_3} \norm{\phi}_{p_4}  \text{ for } j=1,3.
}{commutatorStep3Eq3.52}
Now we will provide an estimate for $J^{it}(\phi \psi)$. We see that for $h \in C^\infty(\td)$ we have 
\eqns{
J^{it}(h)(x) = 
 \sum_{\eta \in \mathbb{Z}^d} e^{2 \pi i \langle x , \eta \rangle } \n{1 + |\eta|^2}^{it/2} \hat{h}(\eta).
}
We see that based on Lemma \ref{technicalLemma1} from Appendix II symbol $\n{1 + |\eta|^2}^{it/2}$ fulfils assumptions of Lemma \ref{differentialCondToDifference} and thus the assumption of Theorem \ref{thCoifMay}. Consequently, we have
\eqnsl{
\norm{J^{it}h}_p 
\le C(1+|t|)^k \norm{h}_p .
}{commutatorStep3Eq3.525}
From this we easily get 
\eqnsl{
\norm{J^{it}(\phi \psi)}_p 
\le C(1+|t|)^k \norm{\psi}_{p_3} \norm{\phi}_{p_4}  .
}{commutatorStep3Eq3.53}
Thus we see that based on (\ref{commutatorStep3Eq3.51}), (\ref{commutatorStep3Eq3.52}), (\ref{commutatorStep3Eq3.53}) we have
\eqns{
\norm{\sigma_{2,1}^{it}(D)(\phi, \psi)}_p \le C(1+|t|)^k \norm{\psi}_{p_3} \norm{\phi}_{p_4}.
}
Based on (\ref{sigmaExt1}), (\ref{sigmaExt2}) and (\ref{commutatorStep3Eq3.525}) we have  
\eqns{
& \norm{\overline{\sigma}_{2,1}^{it}(D)(\phi, \psi)}_p 
 =
\norm{\sigma_{2,1}^{it}(D)(J^{-it} \phi, \psi)}_p
\le 
 C(1+|t|)^{2k} \norm{\psi}_{p_3} \norm{\phi}_{p_4}.
}
In order to use Theorem \ref{GrafacosInterpolationTh} we need to show that family of operators $ \left\{ \overline{\sigma}_{2,1}^z(D) \right\}_{z \in \overline{S}}$ is admissible analytic family. According to the Definition \ref{analyticFamily} we can verify the conditions for smooth functions (which are dense in $L^p$, $1<p<\infty$). Let us choose two functions $\psi, \phi \in C^\infty(\td)$. We clearly see that 
\eqns{ 
S \ni z \longmapsto 
\sum\limits_{\substack{\xi \in \mathbb{Z}^d: |k|<n \\ \eta \in \mathbb{Z}^d: |k|<n}}
e^{2 \pi i x (\xi + \eta)}
\n{1 + |\xi + \eta|^2}^{z/2}\n{1 + |\xi|^2}^{-z/2}
\Phi_2\n{ \frac{1 + |\xi|^2}{1 + |\eta|^2} } 
\hat{\phi}(\xi) \hat{\psi}(\eta)
}
is analytic, because functions of type $S \ni z \mapsto \beta^{z/2}$, $\beta\in\mathbb{R}_+$ are analytic. We will show that the expression on the right-hand side converges uniformly. Indeed, we see that using (\ref{fourierTransformDecay}) we have
\eqnsl{
&  
\sum_{\xi, \eta \in \mathbb{Z}^d} \left|
e^{2 \pi i x (\xi + \eta)}
\n{1 + |\xi + \eta|^2}^{z/2}\n{1 + |\xi|^2}^{-z/2}
\Phi_2\n{ \frac{1 + |\xi|^2}{1 + |\eta|^2} } 
\hat{\phi}(\xi) \hat{\psi}(\eta)
\right| 
\\ & \quad
\le
C \sum_{\xi \in \mathbb{Z}^d} \left| \hat{\phi}(\xi) \right|  \sum_{\eta \in \mathbb{Z}^d} 
\left|\hat{\psi}(\eta) \right|
\le 
C \sum_{\xi \in \mathbb{Z}^d} \frac{C_{\phi}}{(1 + |\xi|^2)^{\frac{d+1}{2}}}  
  \sum_{\eta \in \mathbb{Z}^d} \frac{C_{\psi}}{(1 + |\xi|^2)^{\frac{d+1}{2}}} 
< \infty.
}{analyticCheckEq1}
Thus it is easy to see that
\eqnsl{
\sum\limits_{\substack{\xi \in \mathbb{Z}^d: |k|<n \\ \eta \in \mathbb{Z}^d: |k|<n}}
e^{2 \pi i x (\xi + \eta)}
\n{1 + |\xi + \eta|^2}^{z/2}\n{1 + |\xi|^2}^{-z/2}
&
\Phi_2\n{ \frac{1 + |\xi|^2}{1 + |\eta|^2} } 
\hat{\phi}(\xi) \hat{\psi}(\eta)
\\ & 
\overset{n \rightarrow \infty} 
\rightrightarrows 
\overline{\sigma}_{2,1}^z(D)(\phi,\psi).
}{analyticCheckEq1.5}
Thus, we can conclude that $\overline{\sigma}_{2,1}^z(D)(\phi,\psi)$ is analytic for any $\phi,\psi \in C^\infty(\td)$. Using the same approach we can show continuity of $\overline{S} \ni z  \mapsto \overline{\sigma}_{2,1}^z(D)(\phi,\psi)$. 
We will only apply Theorem \ref{GrafacosInterpolationTh} to one of the variable of $\overline{\sigma}_{2,1}^z(D)(\phi,\psi)$. To show that condition (\ref{thGrafGrowthConOnLine}) holds, we verify that $C(1+|t|)^{2k} \norm{\psi}_{p_3} \in L^1(P_j(\theta, \cdot) dt)$ for $j=0,1$ (interpolation with respect to the first variable). It is obvious based on Remark \ref{pjDecayRate}. Thus using Theorem \ref{GrafacosInterpolationTh} we can deduce that for $0 \le s \le 2k$ the following holds
\eqns{
\norm{\overline{\sigma}_{2,1}^s(D)(\phi, \psi)}_p 
\le 
 C \norm{\psi}_{p_3} \norm{\phi}_{p_4} .
}
Now recalling (\ref{commutatorStep3Eq1}), (\ref{commutatorStep3Eq3}) and (\ref{sigmaExt1}) we have 
\eqns{
\norm{\sigma_{2,1}(D)(f,g)}_p
=
\norm{\overline{\sigma}_{2,1}^s(D)(J^{s} f,g)}_p 
\le 
 C \norm{g}_{p_3} \norm{J^s f}_{p_4} .
}
The validity of the above inequality in case $s/2 > k$ was already justified in reasoning that lead to (\ref{commutatorStep3Eq3.5}). Thus using the above and (\ref{commutatorStep3Eq2.1}) we obtain
\eqnsl{
\norm{\sigma_{2}(D)(f,g)}_p
\le 
 C \norm{g}_{p_3} \norm{J^s f}_{p_4} .
}{commutatorStep3EqFinal}

\subsubsection{Conclusion}
By combining (\ref{X3__}), (\ref{commutatorStep1EqFinal}), (\ref{commutatorStep2EqFinal}) and (\ref{commutatorStep3EqFinal}) we get 
\eqns{
\norm{J^s(f g) - f J^s(g)}_p 
\le C \n{ 
\norm{\partial f}_{p_1} \norm{J^{s-1}g}_{p_2} 
+ 
\norm{g}_{p_3} \norm{J^{s}f}_{p_4} 
}.
}
\end{proof}

\section{Appendix II}
The following lemmas were used in the proof of Lemma \ref{Lem7}.
\begin{lem} \label{ineqalityle1}
Let $\xi$, $\eta \in \mathbb{R}^d$ such that $\frac{1+|\xi|^2}{1 + |\eta|^2} < \frac{1}{9}$, then
\eqnsl{
\left|(1+|\eta|^2)^{-1} \langle \xi, \xi + 2\eta \rangle \right|  < \frac{7}{9}.
}{ineqalityle1A}
\end{lem}
\begin{proof}[Proof of Lemma \ref{ineqalityle1}]
We have
\eqns{
\left|(1+|\eta|^2)^{-1} \langle \xi, \xi + 2\eta \rangle \right| 
& \le 
\frac{|\xi|^2 + 2 |\langle \xi, \eta \rangle|}{1+|\eta|^2}
 \le 
\frac{|\xi|^2}{1 + |\eta|^2}
+ 2
\frac{|\xi|}{\sqrt{1 + |\eta|^2}}
\frac{|\eta|}{\sqrt{1 + |\eta|^2}}
}
Using the fact that $\frac{1+|\xi|^2}{1 + |\eta|^2} < \frac{1}{9}$, we have
\eqns{
\left|(1+|\eta|^2)^{-1} \langle \xi, \xi + 2\eta \rangle \right| 
 <
\frac{1}{9} + 2 \cdot \sqrt{\frac{1}{9}} \cdot 1  
= \frac{7}{9}.
}
\end{proof}
\begin{lem} \label{technicalLemma0}
Let $N \in \mathbb{N}$, $d \in \mathbb{N}_+$. Suppose that $\sigma_1, \sigma_2:\mathbb{R}^d \rightarrow \mathbb{C}$ are two symbols satisfying
\eqnsl{
\left| 
\partial_\xi^\alpha  \sigma_i(\xi) 
\right|
\le \frac{C_{\alpha}^i F_i(\xi)}{\left( 1 + |\xi|^2 \right)^{\frac{|\alpha|}{2}}}
~~~~~\forall \xi \in U \subset \mathbb{R}^d
}{technicalLemma0Eq0}
for all $\alpha=(\alpha_1, \dots, \alpha_d) \in \mathbb{N}^d $ such that $|\alpha| \le N$ and $F_i: \mathbb{R}^d \rightarrow \mathbb{R}_{\geq 0}$. Let us define $\sigma = \sigma_1 \sigma_2$. Then for all $\alpha \in \mathbb{N}^d $  such that $|\alpha| \le N$ there exists constant $C_\alpha$ such that 
\eqns{
\left| 
\partial_\xi^\alpha  \sigma(\xi) 
\right|
\le \frac{C_{\alpha}F_1(\xi)F_2(\xi)}{\left( 1 + |\xi|^2 \right)^{\frac{|\alpha|}{2}}}
~~~~~\forall \xi \in U \subset \mathbb{R}^d.
}
\end{lem}
\begin{proof}[Proof of Lemma \ref{technicalLemma0}]
Let us set $\alpha=(\alpha_1, \dots, \alpha_d) \in \mathbb{N}^d $ such that $|\alpha| \le N$. Now we calculate
\eqns{
\frac{\partial^{|\alpha|} [\sigma_1 \sigma_2]}{\partial_{\xi_1}^{\alpha_1} \dots \partial_{\xi_d}^{\alpha_d}}
= 
\sum_{k_1 = 0}^{\alpha_1} \dots \sum_{k_d = 0}^{\alpha_d}
\binom{\alpha_1}{k_1} \dots \binom{\alpha_d}{k_d}
\frac{\partial^{\sum_{i=1}^d k_i} [\sigma_1]}{\partial_{\xi_1}^{k_1} \dots \partial_{\xi_d}^{k_d}}
\frac{\partial^{|\alpha| - \sum_{i=1}^d k_i} [\sigma_2]}{\partial_{\xi_1}^{\alpha_1 - k_1} \dots \partial_{\xi_d}^{\alpha_d - k_d}}.
}
Thus we have using the assumption (\ref{technicalLemma0Eq0}) we get
\eqns{
\left|
\frac{\partial^{|\alpha|} [\sigma_1 \sigma_2]}{\partial_{\xi_1}^{\alpha_1} \dots \partial_{\xi_d}^{\alpha_d}}
\right|
& \le 
\sum_{k_1 = 0}^{\alpha_1} \dots \sum_{k_d = 0}^{\alpha_d}
\binom{\alpha_1}{k_1} \dots \binom{\alpha_d}{k_d}
\frac{C_{\alpha_1, \dots, \alpha_d}F_1(\xi)}{\left( 1 + |\xi|^2 \right)^{\frac{\sum_{i=1}^d k_i}{2}}}
\frac{C_{\alpha_1 - k_1, \dots, \alpha_d - k_d}F_2(\xi)}{\left( 1 + |\xi|^2 \right)^{|\alpha| -\frac{\sum_{i=1}^d k_i}{2}}} \\
& = 
\frac{F_1(\xi) F_2(\xi)}{\left( 1 + |\xi|^2 \right)^\frac{|\alpha|}{2}}
\sum_{k_1 = 0}^{\alpha_1} \dots \sum_{k_d = 0}^{\alpha_d}
\binom{\alpha_1}{k_1} \dots \binom{\alpha_d}{k_d}
C_{\alpha_1 - k_1, \dots, \alpha_d - k_d}
C_{\alpha_1, \dots, \alpha_d}.
}
\end{proof}
\begin{lem} \label{technicalLemma1}
Let $s \in \mathbb{C}$, $m \in \mathbb{N}$ and $\left\{ \alpha_{i} \right\}_{i = 1}^{d} \in \mathbb{N}^{d}$ such that $\sum_{i=1}^d \alpha_i = m$. Then, there exists $N \in \mathbb{N}$, $\left\{ \omega_{i,j} \right\}_{i,j = 1}^{N,d} \in \mathbb{N}^{dN}$, $\left\{ k_{i} \right\}_{i = 1}^{N} \in \mathbb{N}^{N}$, $\left\{ C_{i} \right\}_{i = 1}^{N} \in \mathbb{C}^{N}$ and  such that
\eqnsl{
\frac{\partial^{m} \left[ (1+|\eta|^2)^{s} \right]}{\partial^{\alpha_1}_{\eta_1}\dots\partial^{\alpha_d}_{\eta_d} } 
= 
\sum_{i = 1}^{N} C_i (1+|\eta|^2)^{s-k_i}
 \eta_1^{\omega_{i,1}} \dots \eta_d^{\omega_{i,d}},
}{technicalLemma1Eq0}
where $\forall i \in \{ 1, \dots, N \}$  $0 \le k_i \le m$,  $2k_i - \sum_{j=1}^{d} \omega_{i,j} = m$ and $|C_i(s,\alpha_1, \dots, \alpha_d)| \le C(m) \cdot (1 + |s|)^{m}$. Also, we have 
\eqns{
\left| \frac{\partial^{m} \left[ (1+|\eta|^2)^{s} \right]}{\partial^{\alpha_1}_{\eta_1}\dots\partial^{\alpha_d}_{\eta_d} } \right|
\le 
C(m) (1+|s|)^{m} (1+|\eta|^2)^{\re s-\frac{m}{2}} .
}
\end{lem}
\begin{proof}[Proof of Lemma \ref{technicalLemma1}]
We will prove the representation formula (\ref{technicalLemma1Eq0}) using the induction method. Let us observe that 
\eqns{
\frac{\partial \left[ (1+|\eta|^2)^{s} \right]}{\partial_{\eta_i}} 
= 
2s (1+|\eta|^2)^{s-1} \eta_i  
}
and thus formula (\ref{technicalLemma1Eq0}) holds for one differentiation. Now we assume that it holds for a certain number of differentiations and will try to deduce its validity after additional differentiation. Indeed we have
\eqns{
\frac{\partial^{{m+1}} \left[ (1+|\eta|^2)^{s} \right]}{\partial^{\alpha_1}_{\eta_1}\dots\partial^{\alpha_j+1}_{\eta_j}\dots\partial^{\alpha_d}_{\eta_d} } 
& = \frac{\partial}{\partial_{\eta_{j}}}
\sum_{i = 1}^{N_{m}} C_i (1+|\eta|^2)^{s-k_i}
 \eta_1^{\omega_{i,1}} \dots \eta_d^{\omega_{i,d}}  \\
 & = 
\sum_{i = 1}^{N_{m}} 2 C_i (s-k_i)(1+|\eta|^2)^{s-(k_i+1)}
 \eta_1^{\omega_{i,1}} \dots \eta_j^{\omega_{i,j} + 1} \dots \eta_d^{\omega_{i,d}}  \\
& + \sum_{i = 1}^{N_{m}} C_i (1+|\eta|^2)^{s-k_i}
 \eta_1^{\omega_{i,1}} \dots \eta_j^{\omega_{i,j} - 1} \dots \eta_d^{\omega_{i,d}}  .
} 
We observe that $2(k_i+1) - \sum_{j=1}^{d} \omega_{i,j} - 1 = m+1$ and $|C_i (s-k_i)| \lesssim \widetilde{C}(1+|s|)^{m+1} $. Thus we proved that (\ref{technicalLemma1Eq0}) holds. Now it is easy to verify that 
\eqns{
\left| \frac{\partial^{m} \left[ (1+|\eta|^2)^{s} \right]}{\partial^{\alpha_1}_{\eta_1}\dots\partial^{\alpha_d}_{\eta_d} } \right|
& \le 
C (1+|s|)^{m} (1+|\eta|^2)^{\re s} 
\sum_{i=1}^{N_{m}} 
\frac{|\eta|^{\sum_{j=1}^d\omega{i,j}}}{(1+|\eta|^2)^{k_i}} \\
& \le 
C (1+|s|)^{m} (1+|\eta|^2)^{\re s } 
\sum_{i=1}^{N_{m}} (1+|\eta|^2)^{-k_i + \frac{1}{2}\sum_{j=1}^d \omega{i,j}} \\
& \le 
C (1+|s|)^{m} (1+|\eta|^2)^{\re s-\frac{m}{2}}.
}
\end{proof}

\begin{lem} \label{technicalLemma2}
Let $r \in \mathbb{N_+}$, $m \in \mathbb{N}$ and $\left\{ \alpha_{i} \right\}_{i = 1}^{d}, \left\{ \beta_{i} \right\}_{i = 1}^{d} \in \mathbb{N}^{d}$ such that $\sum_{i=1}^d \left(\alpha_i + \beta_i \right) = m$. Then, there exists $N \in \mathbb{N}$, $\left\{ \omega_{i,j} \right\}_{i,j = 1}^{N,2d} \in \mathbb{N}^{N \times 2d}$, $\left\{ k_{i} \right\}_{i = 1}^{N} \in \mathbb{N}^{N}$, $\left\{ C_{i} \right\}_{i = 1}^{N} \in \mathbb{R}^{N}$ and  such that
\eqnsl{
\frac{\partial^{m} \left[ 
\langle \xi , \xi + 2\eta \rangle^{r-1}  
\left( \xi_k + 2\eta_k \right)
\right]}{\partial^{\alpha_1}_{\xi_1}\dots\partial^{\alpha_d}_{\xi_d} \partial^{\beta_1}_{\eta_1}\dots\partial^{\beta_d}_{\eta_d} } 
= 
\sum_{i = 1}^{N} C_i \langle \xi , \xi + 2\eta \rangle^{r-1-k_i}
 \xi_1^{\omega_{i,1}} \dots \xi_d^{\omega_{i,d}} \eta_1^{\omega_{i,d+1}} \dots \eta_d^{\omega_{i,2d}}.
}{technicalLemma2Eq0}
Moreover, for $i=1,\dots,N$ we have $ 0 \le k_i \le m $, $ r - 1 - k_i \geq 0 $, $2 k_i - \sum_{j=1}^{2d} \omega_{i,j} = m-1$ and $|C_i| \le  C (1+r)^{m}$. Additionally, there exists $C$ independent from $r$ such that for $(\xi,\eta):\frac{1 + |\xi|^2}{1 + |\eta|^2} < \frac{1}{9}$ we have
\eqns{
\left|
\frac{\partial^{m} \left[ \langle \xi , \xi + 2\eta \rangle^{r-1} (\xi_k + 2\eta_k) \right]}{\partial^{\alpha_1}_{\xi_1}\dots\partial^{\alpha_d}_{\xi_d} \partial^{\beta_1}_{\eta_1}\dots\partial^{\beta_d}_{\eta_d} } 
\right|
& \le C(1+r)^m
\left(\frac{7}{9} \right)^{r}  \frac{\left(1 + |\eta|^2 \right)^{r - \frac{1}{2}}}{\left(1 + |\xi|^2 + |\eta|^2  \right)^{\frac{m}{2}}}.
}
\end{lem}
\begin{proof}[Proof of Lemma \ref{technicalLemma2}]
We will prove this statement via induction argument. We see that for one derivative we have 
\eqns{
\frac{\partial \left[ \langle \xi , \xi + 2\eta \rangle^{r-1} (\xi_k + 2\eta_k) \right]}{\partial_{\xi_j}} 
& = 
(r-1)\langle \xi , \xi + 2\eta \rangle^{r-2} (2\xi_j + 2\eta_j) (\xi_k + 2\eta_k)\\
& +
\langle \xi , \xi + 2\eta \rangle^{r-1} \delta_{kj},
}
\eqns{
\frac{\partial \left[ \langle \xi , \xi + 2\eta \rangle^{r-1} (\xi_k + 2\eta_k) \right]}{\partial_{\eta_j}} 
& = 
(r-1)\langle \xi , \xi + 2\eta \rangle^{r-2} 2 \xi_j (\xi_k + 2\eta_k) \\
& +
2\langle \xi , \xi + 2\eta \rangle^{r-1} \delta_{kj}.
}
We see that the above results of differentiations match the form of (\ref{technicalLemma2Eq0}). Now, we conduct the induction step 
\eqns{
& \frac{\partial^{m+1} \left[ \langle \xi , \xi + 2\eta \rangle^{r-1} (\xi_k + 2\eta_k) \right]}{\partial^{\alpha_1}_{\xi_1}\dots \partial^{\alpha_j + 1}_{\xi_j} \dots\partial^{\alpha_d}_{\xi_d} \partial^{\beta_1}_{\eta_1}\dots\partial^{\beta_d}_{\eta_d} } 
 = 
\frac{\partial}{\partial_{\xi_j} } \left( 
\sum_{i = 1}^{N_{m}} C_i \langle \xi , \xi + 2\eta \rangle^{r-1-k_i}
 \prod_{k=1}^d \xi_k^{\omega_{i,k}} \eta_k^{\omega_{i,d+k}}  
 \right) \\
 & \quad =  
\sum_{i = 1}^{N_{m}} 2 C_i (r-1-k_i)
\langle \xi , \xi + 2\eta \rangle^{r-1-(k_i+1)}
(\xi_j + \eta_j) \xi_1^{\omega_{i,1}} \dots \xi_d^{\omega_{i,d}} \eta_1^{\omega_{i,d+1}} \dots \eta_d^{\omega_{i,2d}}  \\
  & \quad \phantom{=} +  
\sum_{i = 1}^{N_{m}} C_i
\langle \xi , \xi + 2\eta \rangle^{r-1-k_i}
 \xi_1^{\omega_{i,1}} \dots \xi_j^{\omega_{i,j}-1} \dots \xi_d^{\omega_{i,d}} \eta_1^{\omega_{i,d+1}} \dots \eta_d^{\omega_{i,2d}}.
}
We see that $ 2(k_i+1) - (\sum_{j=1}^{2d} \omega_{i,j} + 1) = (m + 1) - 1 $ and $ 2 k_i - (\sum_{j=1}^{2d} \omega_{i,j} - 1) = (m + 1) - 1 $ thus postulated equality (\ref{technicalLemma2Eq0}) holds. In the same way, we get equality for $\frac{\partial}{\partial_{\eta_j} }$. \\
Now, we will prove the inequality stated in the lemma. We see that   
\eqns{
& \left|
\frac{\partial^{m} \left[ \langle \xi , \xi + 2\eta \rangle^{r-1} (\xi_k + 2\eta_k) \right]}{\partial^{\alpha_1}_{\xi_1}\dots\partial^{\alpha_d}_{\xi_d} \partial^{\beta_1}_{\eta_1}\dots\partial^{\beta_d}_{\eta_d} } 
\right| \\ & \qquad \qquad
\le C (1+r)^m
\sum_{i = 1}^{N_{m}} |\langle \xi , \xi + 2\eta \rangle|^{r-1-k_i}
 |\xi|^{\sum_{j=1}^{d} \omega_{i,j}} |\eta|^{\sum_{j=n+1}^{2d} \omega_{i,j}}. 
}
We modify the right-hand side in the following way
\eqns{
& \left|
\frac{\partial^{m} \left[ \langle \xi , \xi + 2\eta \rangle^{r-1} (\xi_k + 2\eta_k) \right]}{\partial^{\alpha_1}_{\xi_1}\dots\partial^{\alpha_d}_{\xi_d} \partial^{\beta_1}_{\eta_1}\dots\partial^{\beta_d}_{\eta_d} } 
\right| \\
& ~~ \le C(m) (1+r)^m
 \sum_{i = 1}^{N_{m}} \left( \frac{9}{7} \frac{|\langle \xi , \xi + 2\eta \rangle|}{1 + |\eta|^2} \right)^{r - 1 -k_i} \left(\frac{7}{9}(1 + |\eta|^2) \right)^{r - 1 - k_i}
 |\xi|^{\sum_{j=1}^{d} \omega_{i,j}} |\eta|^{\sum_{j=d+1}^{2d} \omega_{i,j}}.
}
Based on the fact that $r - 1 - k_i \geq 0$ and on Lemma \ref{ineqalityle1}  we have
\eqns{
& \left|
\frac{\partial^{m} \left[ \langle \xi , \xi + 2\eta \rangle^{r-1} (\xi_k + 2\eta_k) \right]}{\partial^{\alpha_1}_{\xi_1}\dots\partial^{\alpha_d}_{\xi_d} \partial^{\beta_1}_{\eta_1}\dots\partial^{\beta_d}_{\eta_d} } 
\right| 
\\ & ~~~~ \le
C(m)(1 + r)^m \left(\frac{7}{9} \right)^{r} \left(1 + |\eta|^2 \right)^{r - \frac{1}{2}} \sum_{i = 1}^{N_{m}} \frac{|\xi|^{\sum_{j=1}^{d} \omega_{i,j}} |\eta|^{\sum_{j=d+1}^{2d} \omega_{i,j}}}{\left( 1 + |\eta|^2 \right)^{k_i+\frac{1}{2}}}.
}
Using $\frac{1 + |\xi|^2}{1 + |\eta|^2} < \frac{1}{9}$ we have
\eqns{
& \left|
\frac{\partial^{m} \left[ \langle \xi , \xi + 2\eta \rangle^{r-1} (\xi_k + 2\eta_k) \right]}{\partial^{\alpha_1}_{\xi_1}\dots\partial^{\alpha_d}_{\xi_d} \partial^{\beta_1}_{\eta_1}\dots\partial^{\beta_d}_{\eta_d} } 
\right| 
\\ & \qquad \le
C(m)(1 + r)^m \left(\frac{7}{9} \right)^{r} \left(1 + |\eta|^2 \right)^{r - \frac{1}{2}}
 \sum_{i = 1}^{N_{m}} 
 \frac{\left(1 + |\xi|^2 + |\eta|^2 \right)^{\frac{1}{2}\sum_{j=1}^{2d} \omega_{i,j}}}{\left(1 + |\xi|^2 + |\eta|^2 \right)^{k_i + \frac{1}{2}}}.
}
Now using the fact that $ 2 k_i - \sum_{j=1}^{2d} \omega_{i,j} = m - 1 $ we get
\eqns{
\left|
\frac{\partial^{m} \left[ \langle \xi , \xi + 2\eta \rangle^{r-1} (\xi_k + 2\eta_k) \right]}{\partial^{\alpha_1}_{\xi_1}\dots\partial^{\alpha_d}_{\xi_d} \partial^{\beta_1}_{\eta_1}\dots\partial^{\beta_d}_{\eta_d} } 
\right|
& \le C(m)(1 + r)^m
\left(\frac{7}{9} \right)^{r} \frac{\left(1 + |\eta|^2 \right)^{r - \frac{1}{2}} }{\left(1 + |\xi|^2 + |\eta|^2  \right)^{\frac{m}{2}}} .
}
\end{proof}

\begin{lem} \label{technicalLemma3}
Let  $m \in \mathbb{N_+}$ and $\left\{ \alpha_{i} \right\}_{i = 1}^{d}, \left\{ \beta_{i} \right\}_{i = 1}^{d} \in \mathbb{N}^{d}$ such that $\sum_{i=1}^d \left(\alpha_i + \beta_i \right) = m$. Let $\Phi \in C^\infty(\mathbb{R})$ such that $ \supp \frac{\partial \Phi}{\partial x} \subset [a,b] $ for some $a, b > 0$. Then we have
\eqnsl{
\left| \frac{\partial^{m} \left[ 
\Phi \left( \frac{1 + |\xi|^2}{1 + |\eta|^2}  \right)
\right]}{\partial^{\alpha_1}_{\xi_1}\dots\partial^{\alpha_d}_{\xi_d} \partial^{\beta_1}_{\eta_1}\dots\partial^{\beta_d}_{\eta_d} } 
\right|
& \le
C \left( 1 + |\xi|^2 + |\eta|^2\right)^{-\frac{m}{2}}
}{technicalLemma3Eq0}
\end{lem}
\begin{proof}[Proof of Lemma \ref{technicalLemma3}]
First, we will show that there exists $N \in \mathbb{N}$, $\left\{ \omega_{i,j} \right\}_{i,j = 1}^{N,2d} \in \mathbb{N}^{N \times 2d}$, $\left\{ k_{i} \right\}_{i = 1}^{N}, \left\{ \kappa_{i,\xi} \right\}_{i = 1}^{N}, \left\{ \kappa_{i,\eta} \right\}_{i = 1}^{N} \in \mathbb{N}^{N}$, $\left\{ C_{i} \right\}_{i = 1}^{N} \in \mathbb{R}^{N}$ and  such that derivatives can be express in the following way:
\eqnsl{
& 
\frac{\partial^{m} \left[ 
\Phi \left( \frac{1 + |\xi|^2}{1 + |\eta|^2}  \right)
\right]}{\partial^{\alpha_1}_{\xi_1}\dots\partial^{\alpha_d}_{\xi_d} \partial^{\beta_1}_{\eta_1}\dots\partial^{\beta_d}_{\eta_d} } \\
& ~~~~~~ = 
\sum_{i = 1}^{N_{m}} C_i \Phi^{(k_i)}\left( \frac{1 + |\xi|^2}{1 + |\eta|^2}  \right)
\frac{( 1 + |\xi|^2)^{ \kappa_{i,\eta}}}{( 1 + |\eta|^2)^{ \kappa_{i,\xi} + \kappa_{i,\eta}}}
\xi_1^{\omega_{i,1}} \dots \xi_d^{\omega_{i,d}} \eta_1^{\omega_{i,d+1}} \dots \eta_d^{\omega_{i,2d}},
}{technicalLemma3Eq1}
where $1 \le k_i \le m$,  $2\kappa_{i,\xi} - \sum_{j=1}^{2d} \omega_{i,j} = m $.
We see that for $m=1$ such representation is valid. Now we assume that formula is valid for some number of differentiations and we will check if the formula is still valid after additional differentiation. We start with differentiation with respect to $\xi_j$:
\eqns{
& \frac{\partial^{m+1} \left[ 
\Phi \left( \frac{1 + |\xi|^2}{1 + |\eta|^2}  \right)
\right]}{\partial^{\alpha_1}_{\xi_1}\dots\partial^{\alpha_j+1}_{\xi_j}\dots\partial^{\alpha_d}_{\xi_d} \partial^{\beta_1}_{\eta_1}\dots\partial^{\beta_d}_{\eta_d} } 
\\ &
= 
\frac{\partial}{ \partial_{\xi_j}}
\left[\sum_{i = 1}^{N_{m}} C_i \Phi^{(k_i)}\left( \frac{1 + |\xi|^2}{1 + |\eta|^2}  \right)
\frac{( 1 + |\xi|^2)^{ \kappa_{i,\eta}}}{( 1 + |\eta|^2)^{ \kappa_{i,\xi} + \kappa_{i,\eta}}}
\xi_1^{\omega_{i,1}} \dots \xi_d^{\omega_{i,d}} \eta_1^{\omega_{i,d+1}} \dots \eta_d^{\omega_{i,2d}}
\right].
}
After caring out the differentiation we get
\eqns{
& \frac{\partial^{m+1} \left[ 
\Phi \left( \frac{1 + |\xi|^2}{1 + |\eta|^2}  \right)
\right]}{\partial^{\alpha_1}_{\xi_1}\dots\partial^{\alpha_j+1}_{\xi_j}\dots\partial^{\alpha_d}_{\xi_d} \partial^{\beta_1}_{\eta_1}\dots\partial^{\beta_d}_{\eta_d} } \\ &
= 
\sum_{i = 1}^{N_{m}} 2 C_i \Phi^{(k_i+1)}\left( \frac{1 + |\xi|^2}{1 + |\eta|^2}  \right)
\frac{( 1 + |\xi|^2)^{ \kappa_{i,\eta}}}{( 1 + |\eta|^2)^{ \kappa_{i,\xi} + \kappa_{i,\eta} + 1}}
\xi_j \prod_{k=1}^d \xi_k^{\omega_{i,k}} \eta_k^{\omega_{i,d+k}} \\
& + 
\sum_{i = 1}^{N_{m}} C_i 2 \kappa_{i,\eta} \Phi^{(k_i)}\left( \frac{1 + |\xi|^2}{1 + |\eta|^2}  \right)
\frac{( 1 + |\xi|^2)^{ \kappa_{i,\eta} - 1}}{( 1 + |\eta|^2)^{ \kappa_{i,\xi} + 1 + \kappa_{i,\eta} - 1}}
\xi_j \prod_{k=1}^d \xi_k^{\omega_{i,k}} \eta_k^{\omega_{i,d+k}} \\
& + 
\sum_{i = 1}^{N_{m}} C_i \omega_{i,j}  \Phi^{(k_i)}\left( \frac{1 + |\xi|^2}{1 + |\eta|^2}  \right)
\frac{( 1 + |\xi|^2)^{ \kappa_{i,\eta}}}{( 1 + |\eta|^2)^{ \kappa_{i,\xi} + \kappa_{i,\eta}}}
\xi_j^{-1} \prod_{k=1}^d \xi_k^{\omega_{i,k}} \eta_k^{\omega_{i,d+k}}
}
Obtained formula matches the structure from equation (\ref{technicalLemma3Eq1}). Now let us check the validity after additional differentiation with respect to $\eta_j$:
\eqns{
& \frac{\partial^{m+1} \left[ 
\Phi \left( \frac{1 + |\xi|^2}{1 + |\eta|^2}  \right)
\right]}{\partial^{\alpha_1}_{\xi_1}\dots\partial^{\alpha_d}_{\xi_d} \partial^{\beta_1}_{\eta_1}\dots\partial^{\beta_j+1}_{\eta_j}\dots\partial^{\beta_d}_{\eta_d} } 
\\ &
= 
\frac{\partial}{ \partial_{\eta_j}}
\left[\sum_{i = 1}^{N_{m}} C_i \Phi^{(k_i)}\left( \frac{1 + |\xi|^2}{1 + |\eta|^2}  \right)
\frac{( 1 + |\xi|^2)^{ \kappa_{i,\eta}}}{( 1 + |\eta|^2)^{ \kappa_{i,\xi} + \kappa_{i,\eta}}}
\xi_1^{\omega_{i,1}} \dots \xi_d^{\omega_{i,d}} \eta_1^{\omega_{i,d+1}} \dots \eta_d^{\omega_{i,2d}}
\right].
}
After caring out the differentiation we get
\eqns{
& \frac{\partial^{m+1} \left[ 
\Phi \left( \frac{1 + |\xi|^2}{1 + |\eta|^2}  \right)
\right]}{\partial^{\alpha_1}_{\xi_1}\dots\partial^{\alpha_d}_{\xi_d} \partial^{\beta_1}_{\eta_1}\dots\partial^{\beta_j+1}_{\eta_j}\dots\partial^{\beta_d}_{\eta_d} } 
\\ &
\quad = 
\sum_{i = 1}^{N_{m}} (-2) C_i \Phi^{(k_i+1)}\left( \frac{1 + |\xi|^2}{1 + |\eta|^2}  \right)
\frac{( 1 + |\xi|^2)^{ \kappa_{i,\eta}+1}}{( 1 + |\eta|^2)^{ \kappa_{i,\xi} + 1 + \kappa_{i,\eta} + 1}}
\eta_j \prod_{k=1}^d \xi_k^{\omega_{i,k}} \eta_k^{\omega_{i,d+k}} \\
& \quad \phantom{=} + 
\sum_{i = 1}^{N_{m}} C_i (-2) (\kappa_{i,\eta} + \kappa_{i,\xi}) \Phi^{(k_i)}\left( \frac{1 + |\xi|^2}{1 + |\eta|^2}  \right)
\frac{( 1 + |\xi|^2)^{ \kappa_{i,\eta}}}{( 1 + |\eta|^2)^{ \kappa_{i,\xi} + 1 + \kappa_{i,\eta}}}
\eta_j \prod_{k=1}^d \xi_k^{\omega_{i,k}} \eta_k^{\omega_{i,d+k}} \\
& \quad \phantom{=} + 
\sum_{i = 1}^{N_{m}} C_i \omega_{i,d+j}  \Phi^{(k_i)}\left( \frac{1 + |\xi|^2}{1 + |\eta|^2}  \right)
\frac{( 1 + |\xi|^2)^{ \kappa_{i,\eta}}}{( 1 + |\eta|^2)^{ \kappa_{i,\xi} + \kappa_{i,\eta}}}
\eta_j^{-1} \prod_{k=1}^d \xi_k^{\omega_{i,k}} \eta_k^{\omega_{i,d+k}}.
}
Again we see that structure of (\ref{technicalLemma3Eq1}) is preserved after differentiation. Now, we can finally prove the estimate (\ref{technicalLemma3Eq0}). First, let us observe that
\eqns{
\left| \frac{\partial^{m} \left[ 
\Phi \left( \frac{1 + |\xi|^2}{1 + |\eta|^2}  \right)
\right]}{\partial^{\alpha_1}_{\xi_1}\dots\partial^{\alpha_d}_{\xi_d} \partial^{\beta_1}_{\eta_1}\dots\partial^{\beta_d}_{\eta_d} } 
\right|
=0 
~~~~~\text{ for } (\xi, \eta): \frac{1 + |\xi|^2}{1 + |\eta|^2} \le b \text{ or } \frac{1 + |\xi|^2}{1 + |\eta|^2} \geq a
}
and thus we will focus on the case when $ a < \frac{1 + |\xi|^2}{1 + |\eta|^2} < b$:
\eqns{
\left| \frac{\partial^{m} \left[ 
\Phi \left( \frac{1 + |\xi|^2}{1 + |\eta|^2}  \right)
\right]}{\partial^{\alpha_1}_{\xi_1}\dots\partial^{\alpha_d}_{\xi_d} \partial^{\beta_1}_{\eta_1}\dots\partial^{\beta_d}_{\eta_d} } 
\right|
& \le
C \sum_{i = 1}^{N_{m}}
\frac{( 1 + |\xi|^2)^{ \kappa_{i,\eta}}}{( 1 + |\eta|^2)^{ \kappa_{i,\xi} + \kappa_{i,\eta}}}
|\xi_1|^{\omega_{i,1}} \dots \xi_d^{\omega_{i,d}} \eta_1^{\omega_{i,d+1}} \dots \eta_d^{\omega_{i,2d}} \\
& \le
C \sum_{i = 1}^{N_{m}}
\frac{( 1 + |\xi|^2 + |\eta|^2)^{\frac{1}{2} \sum_{j=1}^{2d} \omega_{i,j}}}{( 1 + |\eta|^2)^{ \kappa_{i,\xi}}}.
}
Now using the inequality $1 + |\eta|^2 \geq \frac{1}{2} \n{1 + |\eta|^2} + \frac{1}{2b}\n{1 + |\xi|^2} \geq \min \left\{\frac{1}{2}, \frac{1}{2b} \right\} \n{1 + |\eta|^2 + |\xi|^2} $ to obtain
\eqns{
\left| \frac{\partial^{m} \left[ 
\Phi \left( \frac{1 + |\xi|^2}{1 + |\eta|^2}  \right)
\right]}{\partial^{\alpha_1}_{\xi_1}\dots\partial^{\alpha_d}_{\xi_d} \partial^{\beta_1}_{\eta_1}\dots\partial^{\beta_d}_{\eta_d} } 
\right|
& \le
C \sum_{i = 1}^{N_{m}}
\left( 1 + |\xi|^2 + |\eta|^2\right)^{\frac{1}{2}\left(-2 \kappa_{1,\xi} +  \sum_{j=1}^{2d} \omega_{i,j} \right)} \\
& \le
C \left( 1 + |\xi|^2 + |\eta|^2\right)^{-\frac{m}{2}}.
}
\end{proof}
\begin{lem} \label{technicalLemma4}
Let $s \in \mathbb{C}$ such that $\re s \geq 0$, $m \in \mathbb{N_+}$ and $\left\{ \alpha_{i} \right\}_{i = 1}^{d}, \left\{ \beta_{i} \right\}_{i = 1}^{d} \in \mathbb{N}^{d}$ such that $\sum_{i=1}^d (\alpha_i + \beta_i) = m$. Then, there exists $N \in \mathbb{N}$, $\left\{ \omega_{i,j} \right\}_{i,j = 1}^{N,2d} \in \mathbb{N}^{dN}$, $\left\{ k_{i} \right\}_{i = 1}^{N} \in \mathbb{N}^{N}$, $\left\{ C_{i} \right\}_{i = 1}^{N} \in \mathbb{C}^{N}$ and  such that
\eqnsl{
\frac{\partial^{m} \left[ (1+|\xi + \eta|^2)^{s}\right]}{\partial^{\alpha_1}_{\xi_1}\dots\partial^{\alpha_d}_{\xi_d} \partial^{\beta_1}_{\eta_1}\dots\partial^{\beta_d}_{\eta_d} } 
= 
\sum_{i = 1}^{N} C_i(s) (1+|\xi + \eta|^2)^{s-k_i}
 \eta_1^{\omega_{i,1}} \dots \eta_d^{\omega_{i,d}}
 \xi_1^{\omega_{i,d+1}} \dots \xi_d^{\omega_{i,2d}},  
}{technicalLemma4Eq1}
where $\forall i \in \{ 1, \dots, N \}$  $0 \le k_i \le m$,  $2k_i - \sum_{j=1}^{2d} \omega_{i,j} = m$ and $|C_i(s)| \le \overline{C} (1 + |s|)^{k_i}$.
Also, we have 
\eqnsl{
\left| \frac{\partial^{m} \left[ (1+| \xi + \eta|^2)^{s} \right]}{\partial^{\alpha_1}_{\xi_1}\dots\partial^{\alpha_d}_{\xi_d} \partial^{\beta_1}_{\eta_1}\dots\partial^{\beta_d}_{\eta_d} } \right|
\le 
C(1+|s|)^m (1 + |\xi|^2 + |\eta|^2)^{\re s - \frac{m}{2}} 
~~~~
\forall (\xi, \eta) \in U,
}{technicalLemma4Eq2}
where $U = \left\{ (\xi,\eta): \frac{1+|\xi|^2}{1+|\eta|^2} > 9 ~ \text{or} ~ \frac{1+|\xi|^2}{1+|\eta|^2} < \frac{1}{9}  \right\}$. If $\re s - m \geq 0$, then inequality (\ref{technicalLemma4Eq2}) holds for $(\xi,\eta) \in \mathbb{R}^{2d}$.
\end{lem}
\begin{proof} [Proof of Lemma \ref{technicalLemma4}]
The proof of representation formula (\ref{technicalLemma4Eq1}) follows the same as in Lemma \ref{technicalLemma1}, thus we will concentrate only on inequality.
We get  
\eqnsl{
\left| \frac{\partial^{m} \left[ (1+|\xi + \eta|^2)^{s} \right]}{\partial^{\alpha_1}_{\xi_1}\dots\partial^{\alpha_d}_{\xi_d} \partial^{\beta_1}_{\eta_1}\dots\partial^{\beta_d}_{\eta_d}}  \right|
& \le 
C (1 + |s|)^m \sum_{i=1}^{N_{m}} 
\frac{\n{1 + |\eta|^2 + |\xi|^2 }^{\re s + \frac{1}{2}\sum_{j=1}^{2d} \omega_{i,j}}}{(1+|\xi + \eta|^2)^{k_i}}.
}{technicalLemma4Eq3}
Let us observe that from $1+|\xi|^2 > 9 \left( 1 + |\eta|^2 \right)$ we can derive $-\frac{1}{3} \sqrt{|\xi|^2 - 8} < -|\eta|$. Thus we have
\eqns{
1 + |\xi + \eta|^2 
& = 1 + |\xi|^2 + |\eta|^2 + 2 \langle \xi , \eta \rangle
\geq 1 + |\xi|^2 + |\eta|^2 - 2 |\xi| |\eta| 
\\ & 
\geq 1 + |\xi|^2 + |\eta|^2 - \frac{2}{3} |\xi| \sqrt{|\xi|^2 - 8} 
\geq \frac{1}{3} \n{1 + |\xi|^2 + |\eta|^2}.
}
Thus we get 
\eqns{
\left| \frac{\partial^{m} \left[ (1+|\xi + \eta|^2)^{s} \right]}{\partial^{\alpha_1}_{\eta_1}\dots\partial^{\alpha_d}_{\eta_d} \partial^{\beta_1}_{\eta_1}\dots\partial^{\beta_d}_{\eta_d} }  \right|
& \le 
C (1 + |s|)^m \sum_{i=1}^{N_{m}} 
\n{1 + |\eta|^2 + |\xi|^2 }^{\re s - \frac{1}{2}(2k_i - \sum_{j=1}^{2d} \omega_{i,j})}.
}
Using the fact that $2k_i - \sum_{j=1}^{2d} \omega_{i,j} = m$, we get the desired inequality. The other case $1+|\xi|^2 < \frac{1}{9} \left( 1 + |\eta|^2 \right)$ is analogous. Now, let us assume that $\re s - m \geq 0$. Thus from (\ref{technicalLemma4Eq1}) we have
\eqns{ \hspace{-2mm}
\left|
\frac{\partial^{m} \left[ (1+|\xi + \eta|^2)^{s}\right]}{\partial^{\alpha_1}_{\xi_1}\dots\partial^{\alpha_d}_{\xi_d} \partial^{\beta_1}_{\eta_1}\dots\partial^{\beta_d}_{\eta_d}} 
\right|
\le C (1+|s|)^m
\sum_{i = 1}^{N} (1+|\xi + \eta|^2)^{\re s-k_i}
\n{1 + |\eta|^2 + |\xi|^2 }^{\frac{1}{2}\sum_{j=1}^{2d} \omega_{i,j}} .
}
As $k_i \le m$ we see that $\re s - k_i \geq 0$ and thus we have
\eqns{
\left|
\frac{\partial^{m} \left[ (1+|\xi + \eta|^2)^{s}\right]}{\partial^{\alpha_1}_{\xi_1}\dots\partial^{\alpha_d}_{\xi_d} \partial^{\beta_1}_{\eta_1}\dots\partial^{\beta_d}_{\eta_d}} 
\right|
\le C (1+|s|)^m
\sum_{i = 1}^{N} 
\n{1 + |\eta|^2 + |\xi|^2 }^{\re s - \frac{1}{2}(2k_i - \sum_{j=1}^{2d} \omega_{i,j})}. 
}
Again, using the fact that $2k_i - \sum_{j=1}^{2d} \omega_{i,j} = m$, we obtain the desired inequality.
\end{proof}

\end{document}